\documentclass[reqno,oneside,12pt]{amsart}

\usepackage[foot]{amsaddr}
\usepackage{amsfonts}
\usepackage{amsmath}
\usepackage{amssymb}
\usepackage{amsthm}
\usepackage{appendix}
\usepackage{bm}
\usepackage{bbm}
\usepackage{blindtext}
\usepackage{caption}
\usepackage{color}
\usepackage{dsfont}
\usepackage{enumitem}
\usepackage{etoolbox}
\usepackage{fancyhdr}
\usepackage[text={6in,8in},centering]{geometry}
\usepackage{mathdots}
\usepackage{mathrsfs}
\usepackage{scrextend}
\usepackage{thmtools}

\usepackage{tikz}
\usepackage{hyperref}

\newtheorem{definition}{Definition}
\newtheorem*{remark}{Remark}
\newtheorem*{reshyp}{Resolvent Hypothesis}
\newtheorem{theorem}{Theorem}
\newtheorem{invprob}{Inverse Problem}
\newtheorem*{invprob*}{Inverse Problem}
\newtheorem{lemma}{Lemma}
\newtheorem{prop}{Proposition}

\newtheorem{ipr}{Inverse Problem}

\newcommand{\bc}{\textsc{bc}{ }}
\newcommand{\bvp}{\textsc{bvp}{ }}
\newcommand{\ep}{\varepsilon}

\newcommand{\dee}{{\rm{\textnormal{d}}}}

\newcommand{\dn}{\textsc{dn}{ }}
\newcommand{\Gc}{{\Gamma_{\rm{\scriptsize{\textnormal{c}}}}}}
\newcommand{\Gi}{{\Gamma_{\rm{\scriptsize{\textnormal{i}}}}}}

\newcommand{\IM}{{\rm{\textnormal{Im}}}}
\newcommand{\Int}{\rm{\textnormal{int}}}
\newcommand{\lambar}{\overline\lambda}       		
\newcommand{\lan}{\langle}

\newcommand{\loc}{\rm{\scriptsize{\textnormal{loc}}}}

\newcommand{\pd}{\partial}
\newcommand{\ph}{\varphi}
\newcommand{\ran}{\rangle}

\newcommand{\RE}{{\rm{\textnormal{Re}}}}
\newcommand{\rstr}{\restriction}
\newcommand{\supp}{\rm{\textnormal{supp}}}
\newcommand{\thet}{\vartheta}
\newcommand{\vb}{\mathbf{b}}

\newcommand{\vh}{\mathbf{h}}
\newcommand{\vx}{\mathbf{x}}

\renewcommand{\emph}{\textsl}

\title[A boundary-singular inverse problem]{A boundary-singular two-dimensional partial data inverse problem}
\author{Freddy J. F. Symons}
\email{\href{mailto:symonsfj@gmail.com}{\nolinkurl{symonsfj@gmail.com}}}


\begin{document}
\maketitle

\begin{abstract}
	We consider uniqueness in an inverse Schr\"odinger problem in a bounded domain in $\mathbb{R}^2$ given the Dirichlet-to-Neumann map on part of the boundary. On the remaining boundary we impose a new type of singular boundary condition with unknown parameter. Owing to recent results on this class of boundary conditions, we discuss the necessity of an extra point condition to well-define the data for the inverse problem. Our results are two-fold. At a \emph{single} frequency the inverse problem displays non-uniqueness, since an unknown boundary condition can spoil ``seeing'' the Schr\"odinger potential \emph{via} the Dirichlet-to-Neumann map. On the other hand, taking as input data the Dirichlet-to-Neumann map at \emph{every} frequency $\lambda\in\mathbb{R}$ for which it is well-defined yields full uniqueness of the potential and all the boundary conditions. We adapt recent methods in related two-dimensional inverse problems and develop new techniques to cope with the singularity in the boundary condition.
\end{abstract}

\tableofcontents

\section{Introduction}
\label{secintro}

Let $\Omega\subset\mathbb{R}^d$ be an open and bounded set with dimension $d\geq2$, whose boundary $\pd\Omega=\overline\Omega\setminus\Omega$ is a connected piecewise $C^1$ manifold. Separate $\pd\Omega$ into two portions $\Gamma\neq\emptyset$ and $\Gc$, each topologically connected, open with respect to the manifold topology, and chosen so that
\begin{align}
\Gamma\cap\Gc=\emptyset\qquad\rm{ and }\qquad\overline{\Gamma\cup\Gc}=\pd\Omega.\nonumber
\end{align}

Let $q\in L^\infty(\Omega)$ and $f\in C^1(\Gc)$ both be real-valued.
Consider for each $\lambda\in\mathbb{C}$ the Dirichlet-to-Neumann (\textsc{dn}) map $\Lambda_{q,f}(\lambda):H^{1/2}(\Gamma)\rightarrow H^{-1/2}(\Gamma)$ mapping $g\mapsto-\pd_\nu u\rstr_{\Gamma}$ where
$\pd_\nu$ denotes the outward directed normal derivative and $u\in L^2(\Omega)$ solves the boundary-value problem (\textsc{bvp})
\begin{equation}
\label{eqnSchrodinger}
	\left\{\begin{array}{ccccc}
		-\Delta u+qu&=&\lambda u&\rm{in}&\Omega,\\
		u+f\pd_\nu u&=&0&\rm{on}&\Gc,\\
		u			&=&g&\rm{on}&\Gamma.
	\end{array}\right.
\end{equation}
The standard inverse Schr\"odinger problem is to try to recover $q$ from the operator $\Lambda_{q,0}(0)$. In this paper we examine uniqueness in the problem of recovery of both $q$ and $f$ when $f$ is non-zero.

Owing to the results in \cite{marlrozen2009} (see also \cite{berrydennis2008,berry2009}), in dimension $d=2$, if $f$ possesses a simple zero (and otherwise has $0$ as an isolated value in its range) then the operator underlying \eqref{eqnSchrodinger}
is not self-adjoint and its eigenvalues fill $\mathbb{C}$. This means that for every $\lambda\in\mathbb{C}$ the solution of \eqref{eqnSchrodinger} is not uniquely specified so the \dn map $\Lambda_{q,f}(\lambda)$ is ill-defined.

One fixes this by imposing a one-parameter point boundary condition (\textsc{bc}) at the zero $\vx_0$ of $f$. A certain class of these \textsc{bc}s---arising from limit-circle considerations of certain ordinary differential operators (\textsc{odo}s)---yields self-adjoint restrictions of the operator associated with \eqref{eqnSchrodinger}. We denote this \bc (see Definition \ref{defsabc} shortly) by
\begin{equation*}
	\beta_\thet[u]=0\quad\bigl(\thet\in(-{\pi/2},{\pi/2})\bigr)
\end{equation*}
for any solution $u$ of \eqref{eqnSchrodinger}. Here $\thet$ is the polar angle at the point $\vx_0$ and $\beta\in\mathbb{R}$ a parameter. We can then well define the new \dn operator $\Lambda_{q,f,\beta}(\lambda):g\mapsto-\pd_\nu u\rstr_{\Gamma}$ for any $g\in H^{1/2}(\Gamma)$ such that $u$ solves
	\begin{equation}
	\label{eqnspectralSchro}
		\left\{\begin{array}{ccccc}
			(-\Delta+q)u	&=&\lambda u&\rm{in}&\Omega, \\
			u-f\pd_\nu u	&=&0		&\rm{on}&\Gc, \\
			\beta_\thet{[u]}&=&0		&\rm{at}&\vx_0,\\
			u				&=&g		&\rm{on}&\Gamma.
		\end{array}\right.
	\end{equation}
	The inverse problem we examine is thus:
\begin{invprob}
	\label{invqfbeta}
		Given the \dn map $\Lambda_{q,f,\beta}(\lambda)$ for some or all $\lambda$, recover the potential $q$ everywhere in $\Omega$, the function $f$ on $\Gc$ and the \bc parameter $\beta$.
\end{invprob}

With a (non-singular) Dirichlet condition on $\Gc$, this class of inverse problem---often in ``conductivity form'' with \dn map $\Lambda^{\rm{\scriptsize{c}}}_{\gamma}:h\mapsto-\gamma\pd_\nu v\rstr_{\Gamma}~\big(h\in H^{1/2}(\Gamma)\big)$ for the \bvp
\begin{equation}
\label{eqncond}
	\left\{\begin{array}{ccccc}
		-\nabla\cdot(\gamma\nabla v)&=&0&\rm{in}&\Omega,\\
		v							&=&0&\rm{on}&\Gc,\\
		v							&=&h&\rm{on}&\Gamma
	\end{array}\right.
	\nonumber
\end{equation}
---has seen plentiful attention. Calder\'on first proposed this problem of electrical prospection in the `80s \cite{calderon1980} inspired by his engineering work for the Argentine state oil company. Initial analysis was in simplified situations \cite{kohnvoge1984,kohnvoge1985}, and the first proof of uniqueness of $\gamma\in C^{\infty}(\overline\Omega)$ in dimension $d\geq3$ was then given in \cite{sylvuhlm1986,sylvuhlm1987}. Arguably most of the important subsequent activity on uniqueness can be found in the references \cite{alessandrini1990,brown1996,browtorr2003,browuhlm1997,bukhguhlm2002,haberman2015,habertatar2013,kenigsjosuhlm2007,nachman1996,paivpanchuhl2003}, covering full data in dimension $d\geq2$ and partial data in dimension $d\geq3$, and culminating in \cite{astapaiv2006,carorogers2016,isakov2006}. For brevity here, we direct to the latter three and \cite[Sec. 2.4 \& Ch. 4]{symons2017} for a more detailed historical exposition.

In the more relevant case of partial data in two dimensions the only progress lies in \cite{imanuhlyama2010,imanuhlyama2011,guilltzou2011,imanuhlyama2015}. Owing to their greatest simplicity we will focus on adapting the methods in \cite{imanuhlyama2010}. We also mention that for the singular case \eqref{eqnspectralSchro} in a symmetric half-disc geometry
uniqueness of a radially symmetric $q\in L^\infty_{\loc}(0,1]$ at the single frequency $\lambda=0$, given known $f$ and $\beta=0$, was proved in \cite[Sec. 4]{browmarlsymo2016}.

\section{Main results and outline}
\label{secprobdef}

To state our main results precisely we first need to make some definitions.

\begin{definition}[Specially decomposable domain]
\label{defdomain}
	We call a bounded two-dimensional simply connected open $\Omega$ a \emph{specially decomposable domain} if it has a $C^\infty$ boundary and can be written as
	\begin{align}
		\Omega=\Int\big(\overline{\Omega_1\cup\Omega_0}\big),
			\nonumber
	\end{align}
	where $\Omega_1$ is a half-disc of radius $1$ whose straight edge---i.e., diameter---is denoted $\Gamma_1$ and is contained in $\pd\Omega$.
\end{definition}
\begin{definition}[Boundary accessibility]
\label{defaccessibility}
	Letting $\Omega$ be a specially decomposable domain, specify the further boundary decomposition $\pd\Omega=\overline{\Gamma\cup\Gc}$, $\Gamma\cap\Gc=\emptyset$ such that the diameter $\Gamma_1\subset\Gc$, and $\Gamma$ and $\Gc$ are both relatively open. The boundary portions $\Gamma$ and $\Gc$ are called, respectively, \emph{accessible} and \emph{inaccessible}.
\end{definition}
\begin{remark}
	For convenience the coordinates of $\mathbb{R}^2$ are specified so that $\Gamma_1$ aligns with the $y$-axis, the mid-point of $\Gamma_1$ is $0$, and the subdomain $\Omega_1$ is in the right half-plane.
	Moreover we will denote by $\Gamma_0$ the (possibly disjoint) portion of the boundary given by $\Gc\setminus\Gamma_1$, and by $\Gi$ the \emph{interface} between $\Omega_0$ and $\Omega_1$, i.e., $\Gi = \left(\overline{\Omega_1}\cap\overline{\Omega_0}\right)\setminus\pd\Omega$.
\end{remark}

	\begin{center}
	\begin{figure}[h]
		\centering
		\begin{tikzpicture}
		\draw [-,black] (-2,1) to[out=90,in=135] (0,1) to[out=-45,in=180] (2,1) to[out=0,in=20] (2,-1) to[out=200,in=270] (-2,-1) to (-2,1);
		\draw [-,black,domain=-90:90,shift={(-2cm,0cm)}] plot ({cos(\x)}, {sin(\x)});
		\draw [thick,black] (-2.1,1) to (-1.9,1);
		\draw [thick,black] (-2.1,-1) to (-1.9,-1);
		\draw [thick,black] (2,1.1) to (2,0.9);
		\draw [thick,black] (2.04,-1.08) to (1.96,-0.92);
		\fill (-2,0) circle (1pt);
		\node[right] at (-1,0) {$\Gi$};
		\node[left] at (-2,0) {$0$};
		\node[left] at (-2,0.75) {$\Gamma_1$};
		\node[above right] at (0,1) {$\Gamma_0$};
		\node[below right] at (0,-1.6) {$\Gamma_0$};
		\node[right] at (2.4,1) {$\Gamma$};
		\node[right] at (-2,0.5) {$\Omega_1$};
		\node[below left] at (0,1) {$\Omega_0$};
		\end{tikzpicture}
		\caption{an example domain $\Omega=$ int$(\overline{\Omega_1\cup\Omega_0})$}
		\label{figdomain}
	\end{figure}
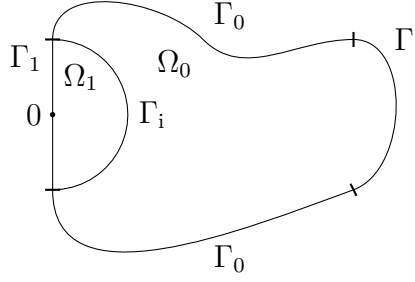
	\end{center}

\begin{definition}[Singular boundary condition]
\label{defsingbc}
	The \emph{Berry--Dennis boundary condition} on $\Omega$ is the requirement on a given $u:\Omega\rightarrow\mathbb{C}$ that
	\begin{equation}
	\label{eqnBDbc}
		(u-f\pd_\nu u)\rstr_{\Gc}=0,
	\end{equation}
	where $f$ is a bounded, real-valued, a.e. absolutely continuous function possessing only one strictly simple zero at the point $0$, and satisfying that for every $\vx\in\Gc\setminus\{0\}$ and $\ep>0$ we have $0$ either not in the range of $f\rstr_{\Gc\setminus\{\vx~:~|\vx|<\ep\}}$, or an isolated value in this range. For the same reasons as in \cite{marlrozen2009} we need $f$ to be linear on the straight edge $\Gamma_1$: we assume $\exists{b}>0$ such that
	\begin{equation}
	\label{eqnfgamma1}
		f(y)=-{b} y\qquad\big(y\in(-1,1)\big),
	\end{equation}
	enforcing a simple zero of $f$ at $0$. Such an $f$ is called an \emph{admissible boundary function}.\footnote{In \cite{marlrozen2009} the parameter ``$\ep$'' is used instead of $b$. We are not interested in its limit approaching $0$, so we use the latter notation.}
\end{definition}

From this definition it is clear that such an $f$ may discontinuously become $0$ or be identically $0$ in a connected subset of $\Gc$.

	\begin{definition}[Self-adjoint boundary condition]
	\label{defsabc}
	
		Let $r=|\vx|,~\thet=\arg(\vx)$ be the usual polar coordinates about $0$. Denoting by $[\cdot,\cdot]$ the Lagrange bracket in $\mathbb{R}^2$, we define an \emph{admissible self-adjoint boundary condition}, with real parameter $\beta$, by
		\begin{align}
		\label{eqnbetabc}
			\beta_\thet[u] &:= \lim_{r\rightarrow0}[u,u_0+\beta v_0](r,\thet)\nonumber\\&=\lim_{r\rightarrow0}\big(u(r,\thet)r\pd_r(u_0+\beta v_0)(r,\thet)-(r\pd_r u(r,\thet))(u_0+\beta v_0)(r,\thet)\big)\nonumber\\&=0,
		\end{align}
		where $u_0$ and $v_0$ are the solutions to
		\begin{align}
			\left\{\begin{array}{ccccc}
			-\Delta u	&=&0&\rm{in }&\Omega_1, \\
			u-f\pd_\nu u&=&0&\rm{on }&\Gamma_1
			\end{array}\right.\nonumber
		\end{align}
		given in \textnormal{\eqref{eqnu0v0}} and $f$ is an admissible boundary function.
	\end{definition}
	\begin{definition}[Admissible potential]
		A function $q\in L^\infty(\Omega)$ is called an \emph{admissible potential}
		if it is locally radially symmetric about $0$, i.e., there is a $\delta>0$ such that in the ball $r<\delta$ we have $\pd_\thet q(r,\thet)=0$. Without loss of generality we can assume $q\rstr_{\Omega_1}$ is radially symmetric, since we can rescale the coordinates so that $\delta=|\Gamma_1|/2=1$.
	\end{definition}
	
Our main results are as follows:	
	\begin{theorem}[Conditional uniqueness at one frequency]
	\label{thmquniqsingbdry}

		Let $\alpha>0$, $\beta\in\mathbb{R}$, and $\Omega$ be specially decomposable.
		\begin{enumerate}[label=(\roman*)]
		\item Suppose $q_1,q_2\in C^{2+\alpha}(\overline\Omega)$ are admissible potentials with $q_1=q_2$ in a neighbourhood of the boundary $\pd\Omega$ and $f$ is an admissible boundary function. If the \dn maps
		are equal at the frequency $\lambda=0$---i.e., $\Lambda_{q_1,f,\beta}(0)=\Lambda_{q_2,f,\beta}(0)$---then $q_1=q_2$ in all of $\Omega$.
		\item Conversely, take admissible boundary functions $f_1$ and $f_2$. Let $q\in C^{2+\alpha}(\overline\Omega)$ be an admissible potential. If $\Lambda_{q,f_1,\beta}(0)=\Lambda_{q,f_2,\beta}(0)$ then $f_1=f_2$.
		\end{enumerate}
	\end{theorem}
	\begin{remark}
		This immediately implies that if the \dn maps $\Lambda_{q_j,f_j,\beta}(0)$ are equal and $q_1=q_2$ in a neighbourhood of $\pd\Omega$ then $q_1=q_2$ everywhere if and only if $f_1=f_2$ everywhere, i.e., the \bc may cloak the potential. 
	\end{remark}
	\begin{theorem}[Uniqueness at all frequencies]
	\label{thmuniqallfreq}
		Let $\Omega$ be specially decomposable, $\alpha>0$,
		$q_1,q_2\in C^{2+\alpha}(\overline\Omega)$ admissible potentials
		that are equal in some neighbourhood of the boundary $\pd\Omega$,
		$f_1$ and $f_2$ admissible boundary functions \emph{supported in the straight edge} $\Gamma_1$, and $\beta_1,\beta_2\in\mathbb{R}$ admissible self-adjoint \textsc{bc}s at $0$.

		If $\Lambda_{q_1,f_1,\beta_1}(\lambda)=\Lambda_{q_2,f_2,\beta_2}(\lambda)$ at every $\lambda\in\mathbb{R}$ for which both \dn maps are defined then $q_1=q_2$, $f_1=f_2$ and $\beta_1=\beta_2$.
	\end{theorem}
	We will prove Theorem \ref{thmquniqsingbdry} by adapting various existing approaches to the situation with our singular \textsc{bc}. The proof of Theorem \ref{thmuniqallfreq} will then apply this conditional uniqueness after extraction of the parameters $b$ (equivalently $f$) and $\beta$ from the following result.
	\begin{theorem}[Negative eigenvalue asymptotics]
		\label{thmnegevasympsgen}
		Let $\Omega$ be specially decomposable, $q$ an admissible potential, $f$ an admissible boundary function that is supported in $\Gamma_1$, and $\beta\in\mathbb{R}$ parameterise a self-adjoint \textsc{bc} at $0$. Then the self-adjoint operator $T$ underlying \textnormal{\eqref{eqnspectralSchro}} has discrete spectrum accumulating only at $\pm\infty$, and its negative eigenvalues possess the asymptotic expansion
		\begin{align}
		\lambda_n= -{\rm e}^{-2{b}(\thet_0+\tan^{-1}\beta)}{\rm e}^{-2n\pi{b}}\big(1+o(1)\big)\qquad(n\rightarrow-\infty).\nonumber
		\end{align}
		 We define $T$ rigorously in \textnormal{\eqref{eqnopfulldom}}. Here $\thet_0$ is a calculable constant (defined explicitly above equation \textnormal{\eqref{eqnAB}}).
	\end{theorem}
	We prove Theorem \ref{thmnegevasympsgen} in Section \ref{secfullfrequniq}, but we can apply it immediately with Theorem \ref{thmquniqsingbdry} and prove Theorem \ref{thmuniqallfreq}.
	\begin{proof}[Proof of Theorem \ref{thmuniqallfreq}]
	Note that the map $\Lambda_{q_1,f_1,\beta_1}(\lambda)=\Lambda_{q_2,f_2,\beta_2}(\lambda)$ is an \emph{operator-valued Herglotz function} of $\lambda$ (see the proof of Lemma \ref{lemHerg}). Since we know its behaviour on the real line, we know where its poles lie, \emph{ergo} we know where the eigenvalues $\lambda_n$ of $T$ are. In particular, by Theorem \ref{thmnegevasympsgen} we may deduce that, as $n\rightarrow-\infty$, we have
	\begin{align}
	-\frac{\log(-\lambda_n)}{2n\pi}\rightarrow{b}_1={b}_2=:{b},\nonumber
	\end{align}
	which determines $f_1=f_2$ completely on $\Gamma_1$, where the latter are supported. The constant $\thet_0$ is fixed and may in principle be calculated (see Lemma \ref{lemnegevasympssym}), so in turn we may calculate from Theorem \ref{thmnegevasympsgen} that, as $n\rightarrow-\infty$,
	\begin{align}
	-\frac{\log(-\lambda_n)+2{b}(n\pi+\thet_0)}{2{b}}\rightarrow\tan^{-1}(\beta_1)=\tan^{-1}(\beta_2)=:\tan^{-1}(\beta).\nonumber
	\end{align}
	
	Finally, we apply Theorem \ref{thmquniqsingbdry} to the triples $(q_1,f,\beta)$ and $(q_2,f,\beta)$, with equality of the \dn maps for any fixed $\lambda\in\mathbb{R}$, to deduce that $q_1=q_2$ in $\Omega$.
\end{proof}

	The remaining paper is devoted to proving Theorems \ref{thmquniqsingbdry} and \ref{thmnegevasympsgen}; we structure it as follows. It will be useful to recapitulate the ideas of \cite{marlrozen2009}, which we do in Section \ref{secMarlRozen}. Then we describe in Section \ref{secexistingmethods} the approach in \cite{imanuhlyama2010}, since we will adapt the methods. In the next three sections we will prove Theorem \ref{thmquniqsingbdry}. Our proof utilises unique continuation and density arguments---adapted from \cite{browmarlreye2016} and developed in Section \ref{secuniqcont}---applied alongside our version of the weighted sum of values of $q$ \cite[Prop. 4.1]{imanuhlyama2010}, proved in Section \ref{secweightedsum}. Subsequently in Section \ref{secconduniq} we conclude the proof of Theorem \ref{thmquniqsingbdry} on conditional uniqueness. In Section \ref{secinterfaceDNops} we develop some results on the interface \dn maps on $\Gi$. We apply these in Section \ref{secfullfrequniq} to prove the negative eigenvalue asymptotics of Theorem \ref{thmnegevasympsgen}. We offer a short discussion in Section \ref{secdiscussion}.

\section{Summary of differential operators with singular boundary conditions}
\label{secMarlRozen}

When considering inverse problems for partial differential operators (\textsc{pdo}s) it is important for the \textsc{bvp}s to generate self-adjoint operators. This owes to the requirement of well-definedness for the resulting \dn map. If the operator is not self-adjoint then the solution to \eqref{eqnSchrodinger} might not exist or be non-unique, leaving the \dn map ill-defined. If, as desired, the operator is self-adjoint, then it defines a unitary evolution group (see, e.g., \cite{reedsimon1980}), and consequently---except at points in the spectrum---the associated \bvp has a unique solution, non-zero for an inhomogeneous boundary condition. This well defines a \dn map. Physically this is realised by the quantity of \textsc{bc}s. Too many, e.g., simultaneous Dirichlet and Neumann conditions, yield only a symmetric operator. Too few and one loses even symmetry.

The recent \cite{berrydennis2008,marlrozen2009} explored self-adjointness of second-order differential operators in bounded two-dimensional domains, with \textsc{bc}s singular at discrete points. In \cite{marlrozen2009} $\Omega$ had the Glazman decomposition of Definitions \ref{defdomain} and \ref{defaccessibility}.

Upon considering the operator
\begin{align}
	D(T_0)	&=\{u\in L^2(\Omega)~|~\Delta u\in L^2(\Omega),(u+{b} y\pd_\nu
					u)\rstr_{\Gamma_1}=0=u\rstr_{\pd\Omega\setminus\Gamma_1}\}, \nonumber\\
		T_0u&=-\Delta u,
		\label{eqnMarlettaRozenblumoperator}
\end{align}
they showed that despite its seemingly ``complete'' set of \textsc{bc}s, it is not symmetric. Its adjoint is symmetric and has a one-dimensional deficiency space.

Thus one may specify self-adjoint restrictions of $T_0$ by imposing an extra \textsc{bc} at $0$. It turns out that any linear combination of the functions
\begin{equation}
\label{eqnu0v0}
	u_0(r,\thet)={\rm e}^{-\thet/{b}}\sin\big(\log(r)/{b}\big),\qquad
	v_0(r,\thet)={\rm e}^{-\thet/{b}}\cos\big(\log(r)/{b}\big)
\end{equation}
can be used to specify a \textsc{bc} \emph{via} a (two-dimensional) Lagrange bracket. For example in \cite{marlrozen2009} the \textsc{bc} at $0$ is the requirement on functions $u$ that
\begin{align}
	[u,u_0](r,\thet):=r\big(u\pd_ru_0-(\pd_ru)u_0\big)(r,\thet)\rightarrow0\qquad(r\rightarrow0).
	\nonumber
\end{align}

In general one may take any $\beta\in\mathbb{R}$ and then require
\begin{equation}
\label{eqn2DLagrangebracket}
	[u,u_0+\beta v_0](r,\thet)\rightarrow0\qquad(r\rightarrow0).
\end{equation}
The above considerations are the motivation for Definitions \ref{defsingbc} and \ref{defsabc}.
	
	We now briefly rewrite the key points of \cite{marlrozen2009}  incorporating a real-valued radial Schr\"odinger potential $q(r)\in L^\infty(0,1;\dee r)$.

	This will explain the following corollary to \cite{marlrozen2009} and provide us with some useful tools.
	\begin{prop}[Marletta--Rozenblum, 2009]
	\label{propmarlrozenmain}
		Consider the operator
		\begin{align*}
			D(L)&:=\{u\in L^2(\Omega_1)~|~\Delta u\in L^2(\Omega_1),u\rstr_{\Gi}=0=(u-by\pd_\nu u)\rstr_{\Gamma_1}\},\nonumber\\
			Lu	&:=(-\Delta+q)u.
		\end{align*}
		There exist on $L^2(0,1;r\dee r)$ ordinary differential operators $L_n~(n\geq0)$ regular at $1$ for which $L$ is (equivalent to) the direct sum of operators $\oplus_{n=0}^\infty L_n$. The $L_n~(n\geq1)$ are all limit-point at $0$ and self-adjoint; $L_0$ is limit-circle at $0$. The self-adjoint restrictions $L_0'$ of $L_0$ are generated by a one-dimensional Lagrange bracket \bc with real parameter $\beta$. On $L^2(\Omega_1)$ the \textsc{pdo}
		\begin{equation}
		\label{eqnLprimedecomp}
			L':=L_0'\oplus\bigoplus\limits_{n=1}^\infty L_n
		\end{equation}
		is self-adjoint, restricts $L$, and has domain
		\begin{equation}
		 D(L')=\{u\in L^2(\Omega_1)~|~\Delta u\in L^2(\Omega_1),u\rstr_{\Gi}=0=(u-b\pd_\nu u)\rstr_{\Gamma_1}=\beta_{\thet}[u]\}.
		\end{equation}
	\end{prop}
	
	This is proved for $q=0$ in \cite{marlrozen2009}, and since multiplication by $q\in L^\infty(0,1)$ is a bounded self-adjoint operator on both $L^2(\Omega_1)$ and $L^2(0,1;r\dee r)$ it holds for $q\neq0$ \emph{via} \cite[Thm. V.4.3]{kato1995}. For most of the rest of this section we explain \eqref{eqnLprimedecomp}, since the underlying tools will be useful.
	
	In polar coordinates the
 eigenvalue problem for $L-\lambda$ is
	\begin{equation}
	\label{eqnevprobL}
		\left\{\begin{array}{ccccc}
			-\displaystyle{\left(\frac{1}{r}\pd_r\left(r\pd_r\right)-\frac{1}{r^2}\pd_\thet^2\right)u(\cdot~;\lambda)} +q(|\cdot|)u(\cdot~;\lambda)
				&=&\lambda u(\cdot~;\lambda)
					&\rm{in}&\Omega_1, \\
			u(\cdot~;\lambda)
				&=&0
					&\rm{on}&\Gi, \\
			\big({b} \displaystyle{\pd_\thet}u(\cdot~;\lambda)+u(\cdot~;\lambda)\big)\rstr_{\thet=\pm\frac{\pi}{2}}
				&=&0
					&\rm{on}&\Gamma_1.
		\end{array}\right.
	\end{equation}
	Separating the variables we find two ordinary differential problems. Firstly the angular problem
	\begin{equation}
	\left\{\begin{array}{cccc}
		-\Theta''(\thet)
			&=&\mu \Theta(\thet)
				&\big(\thet\in(-\frac{\pi}{2},\frac{\pi}{2})\big), \\
		{b}\Theta'(\thet)+\Theta(\thet)
			&=&0
				&(\thet=\pm\frac{\pi}{2}),
	\end{array}\right.\nonumber
	\end{equation}
	is easily calculated to possess eigenvalues and eigenfunctions
	\begin{align}
		\mu_0&=-\frac{1}{{b}^2},
		\qquad\mu_n
		=n^2\qquad(n\geq1); \nonumber\\
		\Theta_n(\thet)&=\left\{\begin{array}{cc}
		{\rm e}^{-\thet/{b}}							&(n=0),\\
		\cos(n\thet)-(n{b})^{-1}\sin(n\thet)	&(n ~\rm{ even}), \\
		\cos(n\thet)+n{b} \sin(n\thet)		&(n ~\rm{ odd}).
		\end{array}\right.\nonumber
	\end{align}
	Replacing the separation-of-variables parameter $\mu$ by $\mu_n$ then yields the ordinary differential system
	\begin{equation}
	\label{eqnradialODE}
		\left\{\begin{array}{cccc}
			-\displaystyle{\frac{1}{r}}
				\big(r
			R_n'(r;\lambda)
			\big)'
			+q(r)R_n(r;\lambda)+\displaystyle{\frac{\mu_n}{r^2}}R_n(r;\lambda)
				&=&\lambda R_n(r;\lambda)
					&\big(r\in(0,1)\big),\\
			R_n(1;\lambda)
				&=&0.
					&
		\end{array}\right.
	\end{equation}
	Since $R(r)\Theta(\thet)\in L^2(\Omega_1;r\dee r\dee\thet)$ if and only if $R\in L^2(0,1;r\dee r)$ and $\Theta\in L^2(-\frac{\pi}{2},\frac{\pi}{2};\dee\thet)$, we see that the natural Hilbert space over which to consider the radial differential system is the weighted space $L^2(0,1;r\dee r)$. The solutions $R_n$ satisfy the following result.
	\begin{prop}[Marletta--Rozenblum, 2009, p. 4]
	\label{propRn}
		Let $\lambda\in\mathbb{C}$ and $q=0$. For each $n\geq1$ a solution $R_n(\cdot~;\lambda)$ of the differential equation in \textnormal{\eqref{eqnradialODE}} is in $L^2(0,1;r\dee r)$ if and only if it is a constant multiple of the Bessel function
		\begin{align}
			J_n(\sqrt\lambda r)\qquad\big(r\in(0,1)\big).\nonumber
		\end{align}
		The function $R_0(\cdot~;\lambda)$ solves the full \bvp \textnormal{\eqref{eqnradialODE}} in $L^2(0,1;r\dee r)$ with $n=0$ if and only if it is a constant multiple of
		\begin{align}
			Y_{i/{b}}(\sqrt\lambda)J_{i/{b}}(\sqrt\lambda r)-
				J_{i/{b}}(\sqrt\lambda)Y_{i/{b}}(\sqrt\lambda r)\qquad\big(r\in(0,1)\big).\label{eqnBesselcombination}
		\end{align}
	\end{prop}
	\begin{remark}
		Even if $q\neq0$, for each $n\geq1$ and $\lambda\in\mathbb{C}$ the solution $R_n(r;\lambda)$ is asymptotically equal to a constant multiple of $J_n(\sqrt\lambda r)$ as $r\rightarrow0$. The same result fails to hold for $R_0$ and \textnormal{\eqref{eqnBesselcombination}}, as both have zeros accumulating at $0$ which are not, in general, the same sequence.
	\end{remark}

	We may now define explicitly the operators $L_n$ in Proposition \ref{propmarlrozenmain}. For $n\geq0$
	\begin{align}
			 D(L_n)
			&:=
			\biggl\{\ph\in L^2(0,1;r\dee r)~\left|~\frac{1}{r}\big(r\ph'(r)\big)'\right.-\frac{\mu_n}{r^2}\ph(r)\in L^2(0,1;r\dee r),
			\ph(1)=0\biggr\},\nonumber\\
			 L_n\ph(r)&:=-\frac{1}{r}\big(r\ph'(r)\big)'+q(r)\ph(r)+\frac{\mu_n}{r^2}\ph(r)\qquad\big(r\in(0,1)\big).
	\end{align}
	By the orthogonality in $L^2(-\frac{\pi}{2},\frac{\pi}{2})$ of $\Theta_n$, the operator $L$ is equal to the orthogonal sum $\oplus_{n=0}^\infty L_n$. As remarked in \cite[pp. 4--5]{marlrozen2009}, for every $n\geq1$ the operator $L_n$ is of limit-point type at $0$, whilst $L_0$ is of limit-circle

	type at $0$. This means that a self-adjoint restriction of $L_0$ may be found by imposing a \bc at $0$. Assuming $0$ is not in the spectrum of $L_0$, this \bc may be written \cite{niessenzettl1992} as a Lagrange bracket with a linear combination of
	\begin{align}
		u_0(r)=\sin\left(\frac{\log(r)}{{b}}\right)\qquad{\rm{and}}\qquad v_0(r)=\cos\left(\frac{\log(r)}{{b}}\right),\nonumber
	\end{align}
	which are linearly independent in the kernel of $L_0-q$; see \eqref{eqnu0v0}. For technical reasons we always need $u_0$, so we use the combination $u_0+\beta v_0$ for a fixed $\beta\in\mathbb{R}$. Applied to a function $\ph$, the Lagrange bracket \bc requires that as $r\rightarrow0$ we have
	\begin{equation}
		\label{eqnbetabc1}
		[\ph,u_0+\beta v_0](r):=\ph(r)r(u_0+\beta v_0)'(r)-r\ph'(r)(u_0+\beta v_0)(r)\rightarrow0.
	\end{equation}
	Abusing notation we define $\beta[\ph]:=[\ph,u_0+\beta v_0](0^+)$. Thus, our self-adjoint restriction of $L_0$ is
	\begin{align}
		 D(L_0')
		&=
		\biggl\{\ph\in L^2(0,1;r\dee r)~\left|~\frac{1}{r}\left(r\ph'(r)\right)'\right.-\frac{\mu_n}{r^2}\ph(r)\in L^2(0,1;r\dee r),
		\beta{[\ph]}=0=\ph(1)\biggr\},\nonumber\\
		 L_0'\ph(r)
		&=
		-\frac{1}{r}\big(r\ph'(r)\big)'+q(r)\ph(r)+\frac{\mu_n}{r^2}\ph(r)~~~\big(r\in(0,1)\big),\nonumber
	\end{align}
	meaning that the orthogonal sum
	\begin{equation}
		L':=L_0'\oplus\bigoplus\limits_{n=1}^\infty L_n
	\end{equation}
	is a self-adjoint restriction of $L$.

	Applying the asymptotics of Proposition \ref{propRn} we see that the \bc \eqref{eqnbetabc1} when applied to the $\Theta_0$-component of the solution $u$ of \eqref{eqnevprobL} is equivalent to the two-dimensional
	\begin{equation*}
		[u,u_0+\beta v_0](0^+,\thet)=0~~~\big(\thet\in(-\frac{\pi}{2},\frac{\pi}{2})\big).
	\end{equation*}
	With the same abuse of notation we are now justified in defining
	\begin{equation*}
		\beta_\thet[u]=[u,u_0+\beta v_0](0^+,\thet),
	\end{equation*}
	which yields the \bc in \eqref{eqn2DLagrangebracket} and proves that the self-adjoint operator $L'$ has domain
	\begin{equation*}
		D(L')=\{u\in L^2(\Omega_1)~|~\Delta u\in L^2(\Omega_1),u\rstr_{\Gi}=0=(u+{b} y\pd_\nu u)\rstr_{\Gamma_1}=\beta_\thet[u]\}.
	\end{equation*}

	We observe the following corollary to \cite[Sec. 4]{marlrozen2009}.
	\begin{prop}
		\label{propLprimespectrum}
		The operator $L'$ has discrete spectrum accumulating exactly at $\pm\infty$.
	\end{prop}
	\begin{proof}

		It is known \cite[Sec. 4]{marlrozen2009} that $\sigma(L'-q)$ comprises simple eigenvalues accumulating at $\pm\infty$. Now observe that as a simple consequence of the spectral theorem \cite[Sec. VIII.3]{reedsimon1980} $L'-q$ has compact resolvent. Then \cite[Thms. IV.3.17 \& V.4.3]{kato1995} imply that $L'$ also has compact resolvent and is self-adjoint. The last two results combined mean that the spectrum of $L'$ accumulates nowhere except \emph{possibly} at $\pm\infty$. Combining this fact with \cite[Thm. V.4.10]{kato1995}, which here implies

		\begin{equation}
		\label{eqnspectralperturbation}
		\sup\limits_{\lambda\in\sigma(L')}~{\rm\textnormal{dist}}\big(\lambda,\sigma(L'-q)\big)\leq\|q\|_{L^\infty(\Omega_1)},
		\end{equation}
		we obtain sequences of eigenvalues of $L'$ that accumulate at \emph{both} $\pm\infty$.

	\end{proof}
	\begin{remark}
		One cannot guarantee solely from this argument that no new eigenvalues are introduced by adding $q$, which otherwise would alleviate the need for Theorem \ref{thmnegevasympsgen}.
	\end{remark}
	
	By the reasoning in \cite[Sec. 5]{marlrozen2009} one may then define the self-adjoint operator
	\begin{align}
	\label{eqnopfulldom}
		D(T)&:=\{u\in L^2(\Omega)~|~\Delta u\in L^2(\Omega),
		(u-f\pd_\nu u)\rstr_{\Gc}=u\rstr_\Gamma=\beta_\thet[u]=0\},\nonumber\\
		Tu &:= (-\Delta+q)u,
	\end{align}
	also possessing purely discrete spectrum accumulating only at $\pm\infty$, and which is a restriction of the operator underlying \ref{eqnSchrodinger}.
	
	We finish the section with an important fact:

	\begin{reshyp}
		Without a loss of generality, when considering the question of uniqueness outlined in Section \ref{secprobdef} we may assume that both of the operators $T_1$ and $T_2$---corresponding respectively to the triples $(q_1,f_1,\beta_1)$ and $(q_2,f_2,\beta_2)$---have $0$ in their resolvent set. To see why, suppose either (or both) has $0$ in their spectrum. Then we may simply shift their spectra, which are discrete, by adding the same sufficiently small constant to both $q_1$ and $q_2$ so that the resulting operators now both have $0$ in their resolvent set. The question of uniqueness is left unchanged by the shift. We will assume the hypothesis holds for the rest of the paper.
	\end{reshyp}

\section{Complex geometric optics for an existing approach with a Dirichlet \bc}
\label{secexistingmethods}

	To prove Theorem \ref{thmquniqsingbdry} we will adapt some existing methods. To reduce technicality we will use \cite[Thm. 1.1]{imanuhlyama2010} instead of the more powerful results in \cite{guilltzou2011,imanuhlyama2015}. As previously mentioned, all such results employ a Dirichlet \bc on the inaccessible $\Gc$.

	The main ingredients in the proof of \cite[Thm. 1.1]{imanuhlyama2010}
	are Lemma \ref{lemIUYCGO} of this section (merely summarised from \cite{imanuhlyama2010}) and \cite[Prop. 4.1]{imanuhlyama2010}. We will explain the former now. Let any $\vx=(x_1,x_2)\in\mathbb{R}^2$ be represented by the complex number $z=x_1+ix_2$. Notating the partial derivative with respect to $x_j$ as $\pd_j$, define the complex derivatives $\pd_z=(\pd_{1}-i\pd_{2})/2$ and $\pd_{\overline z}=(\pd_{1}+i\pd_{2})/2$. Note that $\Phi(z)~(z\in\Omega)$ is holomorphic if and only if the Cauchy--Riemann equation $\pd_{\overline z}\Phi(z)=0~(z\in\Omega)$ is satisfied.
	\begin{definition}
	\label{defphasefunction}
		We say that a holomorphic $\Phi=\ph+i\psi$ on $\Omega$, with $\ph$ and $\psi$ real-valued and possessing a continuous extension to $\overline\Omega$, is an \emph{admissible phase function} if the following criteria are met:
		\begin{enumerate}[label=(\roman*)]
			\item its set of critical points $\mathcal{H}:=\{z\in\overline\Omega~|~\pd_z\Phi(z)=0\}$ does not intersect $\overline\Gamma$;
			\item its critical points are non-degenerate, i.e., $\pd_z^2\Phi(z)\neq0~(z\in\mathcal{H})$;
			\item its imaginary part $\psi$ vanishes on $\Gc$.
		\end{enumerate}
	\end{definition}
	\begin{remark}
		The function $\Phi$ will be crucial in constructing the complex geometric optics (\textsc{cgo}) solutions. The critical points $\mathcal{H}$ are finite, by holomorphicity of $\Phi$.
	\end{remark}

	Now define the primitive operators $\pd_{\overline{z}}^{-1}$ and $\pd_z^{-1}$ by
	\begin{align*}
		\pd_{\overline z}^{-1}g(z)	:=-\frac{1}{\pi}\int_\Omega
										\frac{g(\xi_1+i\xi_2)}{\xi_1+i\xi_2-z}\dee\xi_2\dee\xi_1
									=:\overline{\pd_z^{-1}\overline{g(z)}}.
	\end{align*}
	\begin{lemma}[Imanuvilov--Uhlmann--Yamamoto, 2010, Sec. 3]
	\label{lemIUYCGO}
		Let $\alpha>0$ and $q_1\in C^{2+\alpha}(\overline\Omega)$, and take any admissible phase function $\Phi$. Then, for each $\tau>0$, there is a solution to
		\begin{align}
			\left\{\begin{array}{rcccc}
				(-\Delta+q_1)u	&=&0&\rm{in}&\Omega, \\
				u				&=&0&\rm{on}&\Gc
			\end{array}\right.\nonumber
		\end{align}
		given by
		\begin{align}
		\label{eqnv1}
		 v_1(z;\tau)={\rm e}^{\tau\Phi(z)}\left(a(z)+a_0(z)/\tau\right)+{\rm e}^{\tau\overline{\Phi(z)}}\left(\overline{a(z)}+\overline{a_1(z)}/\tau\right)
								+{\rm e}^{\tau(\Phi+\overline\Phi)(z)/2}\tilde v_1(x;\tau),
		\end{align}
		where the following conditions hold.
		\begin{enumerate}[label=(\roman*)]
			\item\label{hypafn} The \emph{amplitude function} $a(\cdot)\in C^2(\overline\Omega)$ is non-trivial, holomorphic on $\Omega$, its real part vanishes on $\Gc$ and $a=\pd_z a=0$ in $\mathcal H\cap\pd\Omega$; such an $a(\cdot)$ is called \emph{admissible}.
			\item\label{hypv1tild} The remainder $\|\tilde v_1(\cdot~;\tau)\|_{L^2(\Omega)}=o(1/\tau)$ as $\tau\rightarrow+\infty$.
			\item\label{hypa0a1} The functions $a_0$ and $a_1$ are holomorphic and satisfy the $\tau$-independent \textsc{bc}
			\begin{equation}
				\label{eqnhypa0a1}
					(a_0+a_1)\rstr_{\Gc}=\frac{\tilde M_1}{4\pd_z\Phi}+
													\frac{\tilde M_3}{4\overline{\pd_z\Phi}}.
			\end{equation}
			The functions $\tilde M_1:=\pd_{\overline z}^{-1}(aq_1)-M_1$ and $\tilde M_3(z):=\pd_{z}^{-1}(\overline{a(z)}q_1(z))-M_3(\overline z)$, where $M_1$ and $M_3$ are any polynomials satisfying, for $j=0,1$ and $2$,
			\begin{align}
			\label{eqnhypM1}
				\pd_z^j(\pd_{\overline z}^{-1}(a(z)q_1(z))-M_1(z))	&=0, \\
			\label{eqnhypM3}
				\pd_{\overline z}^j(\pd_z^{-1}(\overline{a(z)}q_1(z))-M_3(\overline z))	&=0.
			\end{align}
		\end{enumerate}
		Moreover for any $q_2\in C^{2+\alpha}(\overline{\Omega})$ and the same $\Phi$ and $\tau$ we can find the same amplitude function $a(\cdot)$ so that
		\begin{align}
		\label{eqnv2}
			 v_2(x;\tau)={\rm e}^{-\tau\Phi(z)}(a(z)+b_0(z)/\tau)+{\rm e}^{-\tau\overline{\Phi(z)}}\left(\overline{a(z)}
						+\overline{b_1(z)}/\tau\right)
						+{\rm e}^{-\tau(\Phi+\overline\Phi)(z)/2}\tilde v_2(x;\tau)
		\end{align}
		solves $(-\Delta+q_2)u=0,u\rstr_{\Gc}=0$, with
		\begin{enumerate}[label=(\roman*),resume]
			\item\label{hypv2tild} $\|\tilde v_2(\cdot~;\tau)\|_{L^2(\Omega)}=o(1/\tau)$ as $\tau\rightarrow+\infty$, and
			\item\label{hypb0b1} holomorphic $b_0$ and $b_1$ satisfying
			\begin{equation}
			\label{eqnhypb0b1}
				(b_0+b_1)\rstr_{\Gc}=-\frac{\tilde M_2}{4\pd_z\Phi}
												-\frac{\tilde M_4}{4\overline{\pd_z\Phi}}.
			\end{equation}
			Here $\tilde M_2:=\pd_{\overline z}^{-1}(aq_2)-M_2$ and $\tilde M_4(z):=\pd_{z}^{-1}(\overline{a(z)}q_2(z))-M_4(\overline z)$, where $M_2$ and $M_4$ are any polynomials satisfying, for $j=0,1$ and $2$,
			\begin{align}
			\label{eqnhypM2}
				\pd_z^j(\pd_{\overline z}^{-1}(a(z)q_2(x))-M_2(z))	&=0, \\
			\label{eqnhypM4}
				\pd_{\overline z}^j(\pd_z^{-1}(\overline a(z)q_2(x))-M_4(\overline z))	&=0.
			\end{align}
		\end{enumerate}
		\end{lemma}

\section{Unique continuation and density results}
\label{secuniqcont}

	To utilise the arguments of \cite{imanuhlyama2010} we will prove density, in the full space of solutions of the differential equation in the \bvp \eqref{eqnSchrodinger}, of solutions that also satisfy the \textsc{bc}s \eqref{eqnBDbc} and \eqref{eqnbetabc}. To do this we will adapt certain results of \cite{browmarlreye2016}. We note that these results are not entirely new to \cite{browmarlreye2016}, being related to early Runge-type theorems, e.g., \cite{rundellstecher1977}.
	The following lemmata achieve unique continuation and the required density. Define $\Omega^*:=\overline\Omega\setminus\{0\}$.

	\begin{lemma}[Unique continuation principles for a Schr\"odinger-type equation]
	\label{lemuniqcont}
	
		Let $\Omega$, $q$ and $f$ be admissible and $\Omega'\subset\Omega$, with $\Omega'$ non-empty, bounded, connected and open, such that $\pd\Omega'\in C^2$ and $\Omega\setminus\overline{\Omega'}$ is connected.
		\begin{enumerate}[label=(\roman{*})]
		
			\item{\label{lemuniqcontball}} If $u\in H^{2}_{\loc}(\Omega^{*})$ satisfies $(-\Delta+q)u=0$ in $\Omega$, and there is a ball $B$ with $\overline B\subset\Omega$ and $u\rstr_B=0$, then $u=0$.
			
			\item\label{lemuniqcontbdry} If $u\in H^{2}_{\loc}(\Omega^{*}\setminus\Omega')$, $(-\Delta+q)u\rstr_{\Omega\setminus\overline{\Omega'}}=0$ and $u\rstr_\Gamma=0=\pd_\nu u\rstr_\Gamma$, then $u\rstr_{\Omega\setminus\Omega'}=0$.
			
		\end{enumerate}
		
	\end{lemma}

	\begin{proof}
	Part \ref{lemuniqcontball} is standard, e.g., \cite[Cor. 1.1]{kurata1997}. Part \ref{lemuniqcontbdry} follows extending $u$ by $0$ through $\Gamma$ and applying \ref{lemuniqcontball}.
	\end{proof}

	\begin{lemma}
	\label{lemdensity}
		Under the hypotheses of Lemma \ref{lemuniqcont} define the sets
		\begin{align*}
			K			&:= \{
								v\in H^{2}_{\loc}(\Omega^*)
									~|~
										(-\Delta+q)v\rstr_\Omega=0
																			\}, \\
			\tilde{K}	&:=	\{
								g\in K
									~|~
										(g-f\pd_\nu g)\rstr_{\Gc}=0=\beta_\thet[g]
																			\}.
		\end{align*}
		Then $\tilde K$ is dense in $K$ under the topology induced by $\|\cdot\|_{L^2(\Omega')}$.
	
	\end{lemma}
	
	\begin{proof}
		
		We adapt the proofs of \cite[Prop. 5.1-2]{browmarlreye2016}. For any measurable $A\subset\mathbb{R}^d$ denote by $\lan\cdot,\cdot\ran_{A}$ the inner product on $L^2(A)$. Let $v\in K$ such that $\lan g,v\ran_{\Omega'}=0$ for every $g\in\tilde K$; we aim to show $v=0$. By the Resolvent Hypothesis we may uniquely define $w\in D(T)$ to solve $Tw=\chi_{\Omega'}\overline v$, where $\chi_A(x)=1~(x\in A);~0~(x\notin A)$.
		
		Now we make some technical definitions (see Fig. \ref{figdomain2}): the sub-domain $\Omega_2\subset\Omega\setminus\overline{\Omega'}$ is taken to have boundary $\pd\Omega_2=\overline{\Gamma\cup\Gamma'\cup\tilde\Gamma}\in C^{2,1}$ such that $\Gamma$, $\Gamma'$ and $\tilde\Gamma$ are all relatively open and disjoint. Here $\tilde\Gamma$ continuously extends $\Gamma$ in $\pd\Omega_2\cap\pd\Omega$ at both its endpoints, $\Gamma'$ is entirely contained in $\Omega\setminus\overline{\Omega'}$, and $0\notin \pd\Omega_2$. This means $\Omega_2$ is a neighbourhood of $\Gamma$, and its complement is separated from $\Gamma$, i.e.,  $\overline{\Omega\setminus\Omega_2}\cap\overline\Gamma=\emptyset$. Take a smooth cut-off function $\mu$ on $\Omega_2$ to be $1$ in a neighbourhood of $\Gamma$ and $0$ in a neighbourhood of $\Gamma'$. Assume without loss of generality that the level curves of $\mu$ are orthogonal to $\tilde\Gamma$.
		
		Consider $g\in\tilde K\subset H^2(\Omega_2)$. We wish to make a decomposition of $g$ into two parts: $g_0\in D(T)$, and $g_1\in H^2(\Omega_2)$ which is supported in $\Omega_2$ and satisfies
		\begin{equation}
		\label{eqng1gonGamma}
			(g_1-f\pd_\nu g_1)\rstr_{\Gamma}=(g-f\pd_\nu g)\rstr_{\Gamma}.
		\end{equation}
		Firstly extend $f$ to a bounded, a.e. continuous function in the interior of $\Omega$. By the trace theorem \cite[Thm. 8.7]{wloka1987} $(g-f\pd_\nu g)\rstr_\Gamma$ can be extended by $0$ to $\tilde F_2:=(g-f\pd_\nu g)\rstr_{\pd\Omega}\in H^{1/2}(\pd\Omega)$. Take any $F_2\in H^{1/2}(\pd\Omega_2)$ which agrees with $\tilde F_2$ on $\Gamma\cup\tilde\Gamma$ and is $0$ on $\Gamma'$. Define $F_1=\mu g\rstr_{\pd\Omega_2}\in H^{3/2}(\pd\Omega_2)$.
		
			\begin{center}
			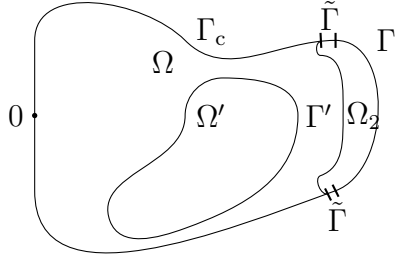
\begin{figure}[hb]
				\centering
				\begin{tikzpicture}
				\draw [-,black] (-2,1) to[out=90,in=135] (0,1) to[out=-45,in=180] (2,1) to[out=0,in=20] (2,-1) to[out=200,in=270] (-2,-1) to (-2,1);
				\draw [thick,black] (2,1.1) to (2,0.9);
				\draw [thick,black] (1.8,1.1) to (1.81,0.9);
				\draw [thick,black] (2.04,-1.08) to (1.96,-0.92);
				\draw [thick,black] (1.94,-1.13) to (1.86,-0.97);
				\fill (-2,0) circle (1pt);
				\node[left] at (-2,0) {$0$};
				\node[right] at (2.4,1) {$\Gamma$};
				\node[below left] at (0,1) {$\Omega$};
				
				\draw [-,black] (0,-1.5) to[out=200,in=290] (-1,-1.4) to[out=110,in=270] (0,0) to[out=90,in=180] (0.5,0.5) to[out=0,in=90] (1.5,0) to[out=270,in=20] (0,-1.5);
				\node[right] at (0,0) {$\Omega'$};
				
				\draw[-,black] (1.9,-1.04) to[out=200,in=200] (1.8,-0.8) to[out=20,in=270] (2.1,0) to[out=90,in=0] (1.8,0.8) to[out=180,in=180] (1.8,0.98);
				\node[right] at(2,0) {$\Omega_2$};
				\node[left] at (2.1,0) {$\Gamma'$};
				\node[above left] at (2.2,1) {$\tilde\Gamma$};
				\node[below left] at (2.3,-1) {$\tilde\Gamma$};
				\node[above right] at (0,0.8) {$\Gc$};
				\end{tikzpicture}
				\caption{an example admissible domain $\Omega$ containing $\Omega'$ and $\Omega_2$}
				\label{figdomain2}
			\end{figure}
			\end{center}

		The inverse trace theorem \cite[Thm. 8.8]{wloka1987} guarantees the existence of $g_1\in H^2(\Omega_2)$ such that $g_1\rstr_{\pd\Omega_2}=F_1\in H^{3/2}(\pd\Omega)$. Furthermore, since
		\begin{align}
			(g_1-F_2)/f\rstr_{\pd\Omega_2}
				&=(\mu g-F_2)/f \nonumber\\
				&=
				\left\{\begin{array}{cccc}
					\pd_\nu g	&\in H^{1/2}(\Gamma),		&\rm{on}&\Gamma,\\
					\mu g/f		&\in H^{3/2}(\tilde\Gamma),	&\rm{on}&\tilde\Gamma,\\
					0										&\in H^{1/2}(\Gamma'),		&\rm{on}&\Gamma'
				\end{array}\right.
				\nonumber
		\end{align}
		is clearly in $H^{1/2}(\pd\Omega_2)$, we can choose this $g_1$ to satisfy $\pd_\nu g_1\rstr_{\pd\Omega_2}=(g_1-F_2)/f\rstr_{\pd\Omega_2}\in H^{1/2}(\pd\Omega_2)$. This ensures that $\pd_\nu g_1\rstr_{\Gamma'}=0$; since $g_1\rstr_{\Gamma'}=F_1\rstr_{\Gamma'}=0$ we may extend this $g_1$ by $0$ into $\Omega$, and note that by construction \eqref{eqng1gonGamma} holds.
		
		Now define $g_0=g-g_1$. By checking the \textsc{bc}s on $\Gamma,\tilde\Gamma$ and $\Gc\setminus\tilde\Gamma$, we see that $g_0\in D(T)$.

		Therefore, for any $g\in\tilde K$,
		\begin{align}
			0	&=\lan g,v\ran_{L^2(\Omega')}
				=\lan g_0,T\overline w\ran_{L^2(\Omega)}+\lan g_1,T\overline w\ran_{L^2(\Omega)}\nonumber\\
					&=\lan Tg_0,\overline w\ran_{L^2(\Omega)}+\int_{\Omega_2}g_1(-\Delta+q)w \nonumber\\
				&=\lan(-\Delta+q)g,\overline w\ran_{L^2(\Omega)}
					+\int_{\Gamma\cup\Gamma'\cup\tilde\Gamma}(-g_1\pd_\nu w+w\pd_\nu g_1)
				=-\int_\Gamma g\pd_\nu w,
		\label{eqnintGamgdnuw}
		\end{align}
		where we used the self-adjointness of $T$, Green's formula, and the fact $g_1,w\in H^2(\Omega_2)$. The final equality was achieved by noting that $w-f\pd_\nu w=0=g-f\pd_\nu g$ on $\tilde\Gamma$, $g=\pd_\nu g=0$ on $\Gamma'$, and $w\in D(T)$ means that $w=0$ on $\Gamma$.
		
		Now observe that the $L^2(\Gamma)$-closure of $\{g\rstr_\Gamma~|~g\in\tilde K\}$ is precisely $L^2(\Gamma)$. This is because for any given basis $\psi_n$ of $L^2(\Gamma)$ we can solve
		\begin{align}
			\left\{\begin{array}{ccccc}
				(-\Delta+q)g&=&0&\rm{in }&\Omega,\\
				g-f\pd_\nu g	&=&0&\rm{on }&\Gc,\\
				\beta_\thet{[g]}	&=&0&\rm{at }&0,\\
				g			&=&\psi_n&\rm{on }&\Gamma,
			\end{array}\right.\nonumber
		\end{align}
		as $0$ is in the resolvent set of $T$. Thus, from \eqref{eqnintGamgdnuw}, we see $\pd_\nu w\rstr_\Gamma=0$. Since $w\rstr_\Gamma=0$, the unique continuation in Lemma \ref{lemuniqcont}\ref{lemuniqcontbdry} implies, from $(-\Delta+q)w\rstr_{\Omega\setminus\overline{\Omega'}}=0$, that $w\rstr_{\Omega\setminus\overline{\Omega'}}=0$.
		
		In particular $w\rstr_{\pd\Omega'}=\pd_\nu w\rstr_{\pd\Omega'}=0$, and so
		\begin{align}
			\lan v,v\ran_{L^2(\Omega')}	&=\int_{\Omega'}v(-\Delta+q)w\nonumber\\
									&=\int_{\Omega'}w(-\Delta+q)v
										+\int_{\pd\Omega'}(-v\pd_\nu w+w\pd_\nu v)
									=0.\nonumber
		\end{align}
		By the unique continuation in Lemma \ref{lemuniqcont}\ref{lemuniqcontball} we deduce that $v\rstr_\Omega=0$.
	\end{proof}

\section{A weighted sum of $q$-values for a singular \bc}
\label{secweightedsum}

	In this section we will prove an analogue to \cite[Prop. 4.1]{imanuhlyama2010} in the case of the \dn map $\Lambda_{q,f,\beta}(0)$ from Inverse Problem \ref{invqfbeta}, as opposed to the more standard \dn map $\Lambda_{q,0}(0)$ used in, e.g., \cite{imanuhlyama2010}. We will use the \textsc{cgo} solutions from Lemma \ref{lemIUYCGO}; by density we can ``approach'' such solutions with those satisfying the singular \textsc{bc}, thanks to Lemma \ref{lemdensity}.

	As in Lemmata \ref{lemuniqcont} and \ref{lemdensity}, we consider $\Omega'\subset\Omega$ to be non-empty, bounded, open and connected subsets of $\mathbb{R}^2$ with $\Omega\setminus\overline{\Omega'}$ connected and $\pd\Omega'\in C^2$. However we are forced, for the same technical reasons as in \cite{imanuhlyama2010}, to accept the restriction $\pd\Omega\in C^\infty$, included as a requirement in our class of admissible domains.

	\begin{prop}
	\label{propwghtdsumq}
	Suppose we have an admissible phase function $\Phi$ (see Definition \textnormal{\ref{defphasefunction}}) and functions $a,a_0,a_1,b_0$ and $b_1$ satisfying \ref{hypafn} (from Lemma \textnormal{\ref{lemIUYCGO}}),\textnormal{\eqref{eqnhypa0a1}} and \textnormal{\eqref{eqnhypb0b1}}, where $M_1,M_2,M_3$ and $M_4$ satisfy \textnormal{\eqref{eqnhypM1}}, \textnormal{\eqref{eqnhypM3}}, \textnormal{\eqref{eqnhypM2}} and \textnormal{\eqref{eqnhypM4}}. Denote by $H_\Phi$ the Hessian matrix of $\Phi$. Let $q_1,q_2\in C^{2+\alpha}(\overline\Omega)$ for some $\alpha>0$, set $q=q_1-q_2$.
	Suppose the \dn maps with the same $f$ and $\beta$ are equal at $\lambda=0$, so $\Lambda_{q_1,f,\beta}(0)=\Lambda_{q_2,f,\beta}(0)$.
	Then, for any $\tau>0$,
	\begin{align}
	\label{eqncompletenessq}
		\sum\limits_{z\in\mathcal{H}}
			&\frac{|a(z)|^2\cos(2\tau\IM[\Phi(z)])}{|\det(\IM[H_\Phi(z)])|^{1/2}}q(z) =
			\nonumber\\
			&=\frac{1}{8\pi}
			\int_\Omega\left[\left(\frac{\tilde M_1-\tilde M_2}{\pd_z\Phi}-4(a_0+b_0)\right)a+
							\left(\frac{\tilde M_3-\tilde M_4}{\overline{\pd_z\Phi}}
								-4\overline{(a_1+b_1)}\right)\overline a\right]q.
	\end{align}

	\end{prop}

	To prove Proposition \ref{propwghtdsumq} we will need an integration-by-parts formula. Since elements of $D(T)$ are not necessarily in $H^2(\Omega)$---in fact, in polar coordinates $\sin({b}^{-1}\log r)+\beta\cos({b}^{-1}\log r)$ may be $C^2$-extended to an element of $D(L')$, and this linear combination clearly fails to be in $H^2$ at $0$---we see that Green's formula cannot be applied. We circumvent this in the following way.

	\begin{lemma}
	\label{lemGreentypeformula}
		Take admissible $\Omega$ and open, connected subset $\Omega'\subset\Omega$. Suppose $\Omega\setminus\overline{\Omega'}$ is connected and $\pd\Omega'\in C^2$. For $j=1,2$ let $q_j$ be admissible potentials with $(q_1-q_2)\rstr_{\Omega\backslash\Omega'}=0$ and $f_j$ admissible boundary functions. Suppose $f_1$ and $f_2$ are identical on $\Gamma_1$, i.e., ${b}_1={b}_2$. Choose an admissible self-adjoint \textsc{bc} $\beta$. Let $u_j$ be any respective solutions to

		\begin{equation}
		\label{eqndoublesystem}
			\left\{\begin{array}{rcccc}
				(-\Delta+q_j)u_j	&=&0	&\rm{in }&\Omega, \\
				u_j-f_j\pd_\nu u_j	&=&0	&\rm{on }&\Gc, \\
				\beta_\thet[u_j]	&=&0	&\rm{at }&0.
			\end{array}\right.
		\end{equation}
		Then, if $\Lambda_{q_1,f_1,\beta}(0)=\Lambda_{q_2,f_2,\beta}(0)$, we have
		\begin{equation}
		\label{eqnpropertyC0}
		\int_\Omega (q_1-q_2)u_1u_2= \int_{\pd\Omega}(u_2\pd_\nu u_1-u_1\pd_\nu u_2).
		\end{equation}
	\end{lemma}

	\begin{proof}
		Take a $\delta$-radius half-disc $\Omega_\delta=\delta\Omega_1\subset\Omega_1$ for some $0<\delta<1$ and define $\Omega_{0,\delta}=\Omega\setminus\overline{\Omega_\delta}$. Without losing generality we may assume $\overline{\Omega_1}\cap\overline{\Omega'}=\emptyset$. Set $\Gamma_\delta=\delta\Gamma_1$ and $\Gamma_{1,\delta}$ to be, respectively, the straight and semi-circular parts of $\pd\Omega_\delta$. Then

		\begin{align}
			\int_\Omega(q_1-q_2)u_1u_2&=
								\underbrace{\int_{\Omega_\delta}(q_1-q_2)u_1u_2}_{\displaystyle{~~~=0}}
											+ \int_{\Omega_{0,\delta}}(-u_1\Delta u_2+u_2\Delta u_1) \nonumber\\
										&= \int_{\pd\Omega_{0,\delta}}(u_2\pd_\nu u_1-u_1\pd_\nu u_2)
											\nonumber\\
										&= \int_{\pd\Omega}(u_2\pd_\nu u_1-u_1\pd_\nu u_2)
											-\int_{\pd\Omega_\delta}(u_2\pd_\nu u_1-u_1\pd_\nu u_2)
		\nonumber
											\nonumber\\
						&= \int_{\pd\Omega}(u_2\pd_\nu u_1-u_1\pd_\nu u_2)
											+\int_{\Gamma_\delta}\left(\frac{1}{f_1}
													-\frac{1}{f_2}\right)u_1u_2
											\nonumber\\
										&\hspace{1cm}-\int_{\Gamma_{1,\delta}}(u_2\pd_\nu u_1-u_1\pd_\nu u_2)
											\nonumber\\
				\label{eqnpropertyC1}
					&= \int_{\pd\Omega}(u_2\pd_\nu u_1-u_1\pd_\nu u_2)
											-\int_{\Gamma_{1,\delta}}(u_2\pd_\nu u_1-u_1\pd_\nu u_2),
		\end{align}
		where we applied identity of the $q_j$ outside $\Omega'$, identity of $f_j$ on $\Gamma_1$, and equality of the \dn maps to eliminate various terms, and Green's formula over $\Omega_{0,\delta}$ to achieve the second line. Thus the lemma follows if the second integral on the right-hand side of \eqref{eqnpropertyC1} converges to 0 as $\delta\rightarrow0$.

		Observe that
		\begin{align}
			\int_{\Gamma_{1,\delta}}(u_1\pd_\nu u_2-u_2\pd_\nu u_1)
				&
				=\int_{-\pi/2}^{\pi/2}\delta(u_1\pd_r u_2-u_2\pd_r u_1)(\delta,\thet)\dee\thet
				\nonumber\\
				&
				=\int_{-\pi/2}^{\pi/2}[u_1,u_2](\delta,\thet)\dee\thet.\nonumber
		\end{align}
		Moreover, by expanding the following determinant and calculating $[v_0,u_0](r,\thet)={b}^{-1}{\rm e}^{-2\thet/{b}}$, one can easily see that
		\begin{equation}
		\label{eqndetu1u2}
			[u_1,u_2](r,\thet)={b} {\rm e}^{2\thet/{b}}\left|\begin{array}{cccc}
				{[u_1,v_0]}&{[u_1,u_0]}\\
				{[u_2,v_0]}&{[u_2,u_0]}
			\end{array}\right|(r,\thet).
		\end{equation}
		Now apply $[u_j,u_0+\beta v_0](r,\thet)\rightarrow0$ as $r\rightarrow0$ to see that the columns in the right-hand side of \eqref{eqndetu1u2} become collinear as $r\rightarrow0$. The lemma follows.

	\end{proof}
	\begin{remark}
		The identity \textnormal{\eqref{eqndetu1u2}} is usually written for solutions of ordinary differential equations; see, e.g., \cite[(2.8-9)]{fulton1977}.
	\end{remark}
	\begin{proof}[Proof of Proposition \ref{propwghtdsumq}]
		We consider all solutions $u_j\quad (j=1,2)$ of \eqref{eqndoublesystem} with $f_j=f$, $\beta_j=\beta$. Define $\Lambda = \Lambda_{q_1,f,\beta}(0)=\Lambda_{q_2,f,\beta}(0)$, $h_1=u_1\rstr_\Gamma$ and $h_2=u_2\rstr_\Gamma$. Also define $\tilde{u}_2\in L^2(\Omega)$ by
		\begin{equation}
			\left\{\begin{array}{rclcc}
			(-\Delta+q_1)\tilde{u}_2 			&=&0&\rm{in}&\Omega,\\
			\tilde{u}_2-f\pd_{\nu}\tilde{u}_2	&=&0&\rm{on}&\Gc,\\
			\beta_\thet{[\tilde{u}_2]}			&=&0&\rm{at}&0,\\
			\tilde{u}_2							&=&h_2&\rm{on}&\Gamma,
			\end{array}\right.
		\end{equation}
		i.e., with potential $q_1$ but \bc $h_2=u_2\rstr_\Gamma$ on $\Gamma$.
		By Lemma \ref{lemGreentypeformula} we see
		\begin{align}
		\label{eqnAlessandriniformula}
			\int_\Omega (q_1-q_2)u_1u_2
				&=\int_{\pd\Omega} (u_2\pd_\nu u_1-u_1\pd_nu u_2)\nonumber\\
				&=\int_\Gamma(h_1\Lambda h_2-h_2\Lambda h_1)\nonumber\\
				&=\int_{\pd\Omega}(\tilde{u}_2\pd_\nu u_1
						-u_1\pd_\nu\tilde{u}_2)\nonumber\\
				&=\int_\Omega(q_1-q_1)u_1\tilde{u}_2=0.
		\end{align}
		Note the hypothesis $q_1-q_2=0$ outside $\Omega'$. Clearly $v_j\in K$---see \eqref{eqnv1} and \eqref{eqnv2}---and $u_j\in\tilde K$, so using Lemma \ref{lemdensity} we deduce from \eqref{eqnpropertyC0} that
		\begin{align}
			\int_\Omega(q_1-q_2)v_1(\cdot~;\tau)v_2(\cdot~;\tau)=0\qquad(\tau>0).\nonumber
		\end{align}
		We have now arrived at precisely \cite[Eq. (4.3)]{imanuhlyama2010}. From here on our proof exactly follows that of \cite[Prop. 4.1]{imanuhlyama2010}.
	\end{proof}

\section{Final steps of proof of Theorem \ref{thmquniqsingbdry}}
\label{secconduniq}

	The proof of Theorem \ref{thmquniqsingbdry} is now straight-forward:
	\begin{proof}[Proof of Theorem \ref{thmquniqsingbdry}]
		To prove that $\Lambda_{q_1,f,\beta}(0)=\Lambda_{q_2,f,\beta}(0)$ and $(q_1-q_2)\rstr_{\Omega\setminus\Omega'}=0$ implies $q_1=q_2$ everywhere, simply apply all but one of the steps in the proof of \cite[Thm. 1.1]{imanuhlyama2010}, replacing \cite[Prop. 4.1]{imanuhlyama2010} with our Proposition \ref{propwghtdsumq}.

		Now for the other claim, namely that $\Lambda_{q,f_1,\beta}(0)=\Lambda_{q,f_2,\beta}(0)$ implies $f_1=f_2$. Let $g\in H^{1/2}(\Gamma)$, and choose functions $u_1$ and $u_2$ solving
		\begin{align}
			\left\{\begin{array}{rclcc}
				(-\Delta+q)u_j		&=&0&\rm{in }&\Omega, \\
				u_j-f_j\pd_\nu u_j	&=&0&\rm{on }&\Gc, \\
				\beta_\thet[u_j]			&=&0&\rm{at }&0, \\
				u_j					&=&g&\rm{on }&\Gamma.
			\end{array}\right.\nonumber
		\end{align}
		By hypothesis $\pd_\nu u_1=\pd_\nu u_2=-\Lambda g$. Then $u_j\in H^{2}_{\loc}(\Omega^*)$, so with the definition $u=u_1-u_2$
		we see
		\begin{align}
			\left\{\begin{array}{rclcc}
				(-\Delta+q)u	&=&0&\rm{in }&\Omega, \\
				u&=&0~=~\pd_\nu u&\rm{on }&\Gamma.
			\end{array}\right.\nonumber
		\end{align}
		The unique continuation principle Lemma \ref{lemuniqcont}\ref{lemuniqcontbdry} immediately implies $u=0$ in $\Omega$, whence $u_1=u_2$ in $\Omega$ so $\pd_\nu u_1=\pd_\nu u_2$ on $\Gc$. Thus, along $\Gc$, $f_1=u_1/\pd_\nu u_1=u_2/\pd_\nu u_2=f_2$.
	\end{proof}

\section{The interface Dirichlet-to-Neumann operators}
\label{secinterfaceDNops}

To prove the asymptotics of Theorem \ref{thmnegevasympsgen} we will find upper and lower asymptotic bounds on the difference between the counting functions of the negative eigenvalues for, respectively, $T$ and $L'$ (see \eqref{eqnLprimedecomp} and \eqref{eqnopfulldom} for the operator definitions). To achieve the upper bound we will consider a pencil of interface \dn operators on $\Gi$. In this section we will develop these operators, establishing results for our proof of Theorem \ref{thmnegevasympsgen}. To ensure sign-definiteness of one of these \dn operators we need the admissible boundary function $f$ to be supported in $\Gamma_1$, justifying the corresponding requirement in both Theorems \ref{thmuniqallfreq} and \ref{thmnegevasympsgen}.

These interface \dn operators are defined as follows: for $j=0,1$,
\begin{align}
\Lambda_j(\lambda):H^{1/2}(\Gi)\ni h\mapsto-\pd_{\nu_j} w_j\in H^{-1/2}(\Gi),\nonumber
\end{align}
where $w_j$ solve the boundary-value problems
\begin{equation}
\label{eqnw0}
\left\{\begin{array}{rclcc}
-\Delta w_0&=	&\lambda w_0	&\rm{in }&\Omega_0, \\
w_0		&=	&0			&\rm{on }&\Gamma_0\cup\Gamma, \\
w_0		&=	&h			&\rm{on }&\Gi,
\end{array}\right.
\end{equation}
\begin{equation}
\label{eqnw1}
\left\{\begin{array}{rclcc}
-\Delta w_1	&=	&\lambda w_1	&\rm{in }&\Omega_1, \\
w_1			&=	&f\pd_{\nu_1}w_1&\rm{on }&\Gamma_1, \\
\beta_\thet[w_1]	&=	&0				&\rm{at }&0, \\
w_1			&=	&h				&\rm{on }&\Gi,
\end{array}\right.
\end{equation}
and $\pd_{\nu_j}$ denotes the outward directed normal derivative for the subdomain $\Omega_j$.

\begin{remark}
	\label{remkereigfn}
	It is clear that a real number $\lambda$ is an eigenvalue for $T$ if and only if the pencil of operators $\Lambda_1(\lambda)+\Lambda_0(\lambda)$ has a non-trivial kernel $\mathcal{K}(\lambda)$, since any function in this kernel will correspond to a pair $(w_0,w_1)$ solving, respectively, \textnormal{\eqref{eqnw0}} and \textnormal{\eqref{eqnw1}}, for which
	$\pd_{\nu_0}w_0=-\pd_{\nu_1}w_1$ on $\Gi$.
\end{remark}
To conduct our analysis we utilise the \emph{normalised} $L^2(-\pi/2,\pi/2)$-basis $\Theta_n$ on $\Gi$ (see Section \ref{secMarlRozen}), defined by
\begin{align}
\Theta_n(\thet)=\left\{\begin{array}{cc}
k_0{\rm e}^{-\thet/{b}}
&(n = 0), \\
k_n(n{b}\cos(n\thet)-\sin(n\thet)	)				&(n ~\rm{ even}), \\
k_n(\cos(n\thet)-n{b}\sin(n\thet)	)				&(n ~\rm{ odd}).
\end{array}\right.\nonumber
\end{align}
where the $k_n$ are chosen so that $\|\Theta_n\|_{L^2(-\pi/2,\pi/2)}=1$.
Since the sufficiently negative eigenvalues of $T$ can only arise from the presence of the singular \textsc{bc} (regular \textsc{bc}s would yield a spectrum that is bounded below), any function $h=\sum_{n=0}^\infty h_n\Theta_n\in\mathcal{K}(\lambda)$ for $\lambda\ll0$ is either identically $0$ or has non-trivial zeroth component.
It follows that for $\lambda\ll0$ such that $\mathcal K(\lambda)\neq\{0\}$ we may normalise the non-trivial kernel element $h$ so that $h_0=1$.

In the basis $\Theta_n$ on $\Gi$ the map $\Lambda_1(\lambda)$ takes the form of an infinite diagonal matrix, represented in the block partitioned form
\begin{equation}
\label{eqnLambda1block}
\left(\begin{array}{c|c}
m_0(\lambda)	&	\mathbf{0}^T \\
\hline
\mathbf{0}	&	M(\lambda)
\end{array}\right).
\end{equation}
Here $m_0(\lambda):=-\ph_0'(1;\lambda)/\ph_0(1;\lambda)\quad(\ph_0\neq0)$ is the Weyl--Titchmarsh $m$-function for the $L^2(0,1)$-limit-circle ordinary differential problem
\begin{align}
\left\{\begin{array}{rclc}
\displaystyle{-\frac{1}{r}\frac{\dee}{\dee r}\left(r\frac{\dee \ph_0}{\dee r}(r;\lambda)\right)
	-\frac{1}{{b}^2r^2}}\ph_0(r;\lambda)	&=&\lambda\ph_0(r;\lambda)
&\big(r\in(0,1)\big), \\
r\big(\ph_0\pd_r(u_0+\beta v_0)-(\pd_r\ph_0)(u_0+\beta v_0)\big)(r;\lambda)
&\rightarrow&0&(r\rightarrow0),
\end{array}\right.\nonumber
\end{align}
the infinite column-vector of zeros is denoted by $\mathbf{0}$, and $M(\lambda)$ is the diagonal submatrix whose $n$-th diagonal term ($n=1,2,3,\ldots$) is the $L^2(0,1)$-limit-point $m$-function $m_n(\lambda):=-\ph_n'(1;\lambda)/\ph_n(1;\lambda)\quad(\ph_n\neq0)$ for the \textsc{ode}
\begin{equation}
\label{eqnnthPLP}
-\frac{1}{r}\frac{\dee}{\dee r}\left(r\frac{\dee \ph_n}{\dee r}(r;\lambda)\right)
+\frac{n^2}{r^2}\ph_n(r;\lambda)	=\lambda\ph_n(r;\lambda) \qquad\big(r\in(0,1)\big).
\end{equation}
In the same basis, $\Lambda_0(\lambda)$ lacks this diagonal structure, although it is symmetric. We may nevertheless use the basis to partition it the same way, labelling it
\begin{align}
\left(\begin{array}{c|c}
a(\lambda)		&	\vb(\lambda)^T \\
\hline
\vb(\lambda)	&	C(\lambda)
\end{array}\right) .\nonumber
\end{align}

\begin{lemma}
	\label{lemHerg}
	~
	
	\begin{enumerate}[label=(\roman*)]
		\item The pencil of \dn maps $\Lambda_1+\Lambda_0$ is a holomorphic, operator-valued function on $\mathbb{C}\setminus\mathbb{R}$.
		\item The derivative of the pencil is a compact operator on the Sobolev space $H^k(\Gi)$ for any $k\geq1/2$.
		\item If the admissible boundary function $f$ is supported in $\Gamma_1$ then the quadratic form $\lan(\Lambda_1+\Lambda_0)(\lambda)h,h\ran_{L^2(\Gi)}$ has imaginary part of the same sign as $\IM(\lambda)$. 
	\end{enumerate}
\end{lemma}
\begin{remark}
	This makes $\Lambda_1+\Lambda_0$ an \emph{operator-valued Herglotz function}.
\end{remark}
\begin{proof}\textit{(i) and (ii).}
	It suffices to show holomorphicity of any \dn map, since the sum of any two will also have the property. Let ${\tilde\Omega}\subset\mathbb{R}^2$ be bounded and simply connected with piecewise $C^2$ boundary $\pd{\tilde\Omega}$, and suppose $\tilde\Gamma$ is a connected subset of $\pd{\tilde\Omega}$, whilst $\tilde\Gamma_c=\pd{\tilde\Omega}\setminus\tilde\Gamma$ is its complement. Denote by $S$ the self-adjoint operator $-\Delta$ with homogeneous Dirichlet conditions on $\tilde\Gamma$ and \emph{any} \textsc{bc} ${\rm{SA}}[\cdot]=0$ on $\tilde\Gamma_c$ that suffices to make $S$ self-adjoint. Then define the \dn operator $\tilde\Lambda(\lambda):h\mapsto\pd_\nu w\rstr_{\tilde\Gamma}$, where
	\begin{equation}
	\label{eqnw}
	\left\{\begin{array}{ccccc}
	-\Delta 	w&=	&\lambda w	&\rm{in }&{\tilde\Omega}, \\
	{\rm{SA}}[w]&=	&0			&\rm{on }&{\tilde\Gamma}_c, \\
	w			&=	&h			&\rm{on }&{\tilde\Gamma}.
	\end{array}\right.
	\end{equation}
	
	If $h\in H^k({\tilde\Gamma})$ for some $k\geq1/2$ then we may choose a $w_0\in L^2({\tilde\Omega})$ taking the value $h$ on ${\tilde\Gamma}$. Then if SA$[w_0]=0$ and $\Delta w_0\in L^2({\tilde\Omega})$, we see that the solution to the boundary-value problem \eqref{eqnw} is given by
	\begin{align}
	w = \big(\mathbbm{1}-(S-\lambda)^{-1}(-\Delta-\lambda)\big)w_0.\nonumber
	\end{align}
	Hence, in terms of the trace maps
	\begin{align}
	D(\gamma_0)&=\{v\in L^2({\tilde\Omega})~|~\Delta v\in L^2({\tilde\Omega}),{\rm{SA}}[v]=0\},
	\qquad
	&\gamma_0\rstr_{C^0({\tilde\Omega})}:v\mapsto v\rstr_{\tilde\Gamma}, \nonumber\\
	D(\gamma_1)&=\{v\in L^2({\tilde\Omega})~|~{\rm{SA}}[v]=0\},
	&\gamma_1\rstr_{C^1({\tilde\Omega})}:v\mapsto\pd_\nu v\rstr_{\tilde\Gamma},\nonumber
	\end{align}
	we see that $\tilde\Lambda(\lambda) = \gamma_1\big(\mathbbm{1}-(S-\lambda)^{-1}(-\Delta-\lambda)\big)\gamma_0^{-1}$, where by $\gamma_0^{-1}$ we mean any right-inverse of $\gamma_0$. Thus we find, \emph{via} the resolvent formula \cite[Thm. VIII.2]{reedsimon1980}, that
	\begin{align}
	\frac{\tilde\Lambda(\lambda)-\tilde\Lambda(\mu)}{\lambda-\mu}
	&=\gamma_1\frac{(S-\lambda)^{-1}(\Delta+\lambda)-
		(S-\mu)^{-1}(\Delta+\mu)}{\lambda-\mu}\gamma_0^{-1}
	\nonumber\\
	&=\gamma_1\frac{\lambda(S-\lambda)^{-1}-\mu(S-\mu)^{-1}+
		\big((S-\lambda)^{-1}-(S-\mu)^{-1}\big)(\Delta)}{\lambda-\mu}\gamma_0^{-1}
	\nonumber
	\\
	\label{eqnDtoNNewt}
	&=\gamma_1(S-\lambda)^{-1}\big(\mathbbm{1}+\mu(S-\mu)^{-1}+(S-\mu)^{-1}(\Delta)\big)
	\gamma_0^{-1}
	\end{align}
	is a smoothing operator of order $-1$ on the scale of Sobolev spaces on ${\tilde\Gamma}$, since it is a product (from right to left) of operators with order $1/2$, $0$, $-2$ and $1/2$. By Sobolev embedding \cite[Thm. V.4.18]{edmundsevans1987} this Newton quotient is compact, and in particular is bounded. As $\mu\rightarrow\lambda$, its norm limit is the compact operator
	\begin{align}
	\tilde\Lambda'(\lambda)=\gamma_1(S-\lambda)^{-1}\big(\mathbbm{1}-
	(S-\lambda)^{-1}(-\Delta-\lambda)\big)\gamma_0^{-1}.\nonumber
	\end{align}
	Of course, this limit is only defined for $\lambda$ in the resolvent set of $S$, which owing to $S=S^*$ contains $\mathbb{C}\setminus\mathbb{R}$.
	\newline
	
	\noindent\textit{(iii).} To establish the Herglotz property, we need to examine the imaginary parts of both $\lan\Lambda_{1}(\lambda)h,h\ran_{\Gi}$ and $\lan\Lambda_{0}(\lambda)h,h\ran_{\Gi}$. The latter involves straight-forward integration by parts. Using Green's formula and the solution $w_0$ to \eqref{eqnw0}, and in particular the homogeneous Dirichlet \bc on $\Gamma\cup\Gamma_0$ (since $f\rstr_{\Gamma_0}=0$), we see
	\begin{align}
	\label{eqnantiHerg0}
	\lan\Lambda_0(\lambda)h,h\ran_{\Gi}
	&=\int_{\Gi}(-\pd_{\nu_0} w_0)\overline{w_0}
	=-\int_{\Omega_0}\big((\Delta w_0)\overline{w_0}+|\nabla w_0|^2\big) \nonumber\\
	&=\lambda\int_{\Omega_0}|w_0|^2-\int_{\Omega_0}|\nabla w_0|^2.
	\end{align}
	This clearly has imaginary part of the same sign as that of $\lambda$.
	
	On the other hand, integration by parts fails for the solutions of \eqref{eqnw1}, since, e.g., the solutions $u_0$ and $v_0$ defined in \eqref{eqnu0v0} are not in $H^2(\Omega)$. Instead, we decompose the solution as in the proof of Lemma \ref{lemGreentypeformula}, effectively treating $-\lambda$ and $-\lambar$ as, respectively, the potentials $q_1$ and $q_2$ (we no longer have equality in a neighbourhood of the boundary). Recall the radial solutions $R_n$ for \eqref{eqnradialODE}; we assume they have unit $L^2$-norm. Then letting $w_1$ solve \eqref{eqnw1}, there are real constants ${c}_n~(n=0,1,2,\ldots)$ so that $W_n(r,\thet)={c}_nR_n(r)\Theta_n(\thet)$ form the terms in a series expansion: $w_1=\sum_{n=0}^\infty W_n=:W_0+W$. The term $W$ arises from the regular part of the problem, and is in $H^2(\Omega_1)$ \cite{marlrozen2009}, and moreover all the $W_n$ are pairwise orthogonal in $L^2(\Omega_1)$. Hence, by Green's formula,
	\begin{align}
	(\lambda-\lambar)\int_{\Omega_1}w_1\overline{w_1}
	&=(\lambda-\lambar)\int_{\Omega_1}W_0\overline{W_0}-
	\int_{\Omega_1}\big((\Delta W)\overline W-W\Delta\overline W\big) \nonumber\\
	&={c}_0^2(\lambda-\lambar)\int_0^1r\dee r~|R_0(r)|^2
	\int_{-\pi/2}^{\pi/2}\dee\thet~|\Theta_0(\thet)|^2
	\nonumber\\
	&\hspace{1cm}-
	\int_{\pd\Omega_1}\big((\pd_{\nu_1}W)\overline W-W\pd_{\nu_1}\overline W\big)
	\nonumber
	\\
	\label{eqnImLambda1}
	&={c}_0^2(\lambda-\lambar)\int_0^1r\dee r~|R_0(r)|^2+\int_{\Gi}\big(W\pd_{\nu_1}\overline W-(\pd_{\nu_1}W)\overline W\big).
	\end{align}
	Without loss of generality we scale ${c}_0$ to be $1$, set $0<\delta<1$ and examine
	\begin{align}
	\label{eqnImLambda11}
	(\lambda-\lambar)\int_\delta^1r\dee r~R_0(r)\overline{R_0(r)}
	&=R_0(1)\overline{R_0'(1)}-R_0'(1)\overline{R_0(1)}-
	\delta\big(R_0(\delta)\overline{R_0'(\delta)}-
	R_0'(\delta)\overline{R_0(\delta)}\big) \nonumber\\
	&=\int_{\Gi}\big(W_0\pd_1\overline{W_0}-(\pd_1W_0)\overline{W_0}\big)-[R_0,\overline{R_0}](\delta).
	\end{align}
	Clearly the lemma will follow if we can show that $[R_0,\overline{R_0}](\delta)$ vanishes as $\delta\rightarrow0$, since by dominated convergence the left-hand side of \eqref{eqnImLambda11} tends to $(\lambda-\overline\lambda)\int_0^1r|R_0(r)|^2\dee r$. Combining this with \eqref{eqnImLambda1} yields
	\begin{align}
	\IM(\lambda)\int_{\Omega_1}|w_1|^2=\IM\left(\lan\Lambda_1(\lambda)h,h\ran_{L^2(\Gi)}\right).\nonumber
	\end{align}
	
	Similarly to Lemma \ref{lemGreentypeformula}, we relate the boundary behaviour of $R_0$ to that of the solutions $u_0$ and $v_0$ for $\lambda=0$ by applying the elementary identity
	\begin{equation}
	\label{eqndetU0}
	[R_0,\overline{R_0}]=-{b}\left|\begin{array}{cc}
	[R_0,u_0] & [R_0,v_0] \\
	{[\overline{R_0},u_0]} & {[\overline{R_0},v_0]}
	\end{array}\right|,
	\end{equation}
	since $u_0$ and $-{b} v_0$ form a fundamental system satisfying $[u_0,v_0]=-{b}^{-1}$. Thus, since both $R_0$ and $\overline{R_0}$
	in its place satisfy $[R_0,u_0+\beta v_0](0^+)=0$, we see that the columns in the right-hand side of \eqref{eqndetU0} become collinear as its argument approaches $0$.
\end{proof}

\begin{lemma}
	\label{lemcorolanalyticFredholm}
	Let the admissible boundary function $f$ be supported in $\Gamma_1$, and let $\lambda$ be less than the infima of the spectra of each $L_j~(j=1,2,3,\ldots$; see Section \ref{secMarlRozen}$)$ and of the Laplace operator in $\Omega_0$ with homogeneous Dirichlet \textsc{bc}s. If $z\in\mathbb{C}_{\rm c}:=\mathbb{C}\setminus[0,+\infty)$ then both $(M+C)(\lambda)$ and $(M+C)(\lambda+z)$ are invertible matrices, and satisfy
	\begin{align}
	[(M+C)(\lambda+z)]^{-1}=\left(\mathbbm{1}+[(M+C)(\lambda)]^{-1}\int_{[\lambda,\lambda+z]}(M+C)'\right)^{-1}
	\nonumber
	[(M+C)(\lambda)]^{-1}.\nonumber
	\end{align}
\end{lemma}
\begin{proof}
	Lemma \ref{lemHerg} implies that $M+C$ is differentiable anywhere in $\lambda+\mathbb{C}_{\rm c}$, and its derivative is compact. By the fundamental theorem of calculus (for operator-valued analytic functions), $(M+C)(\lambda+z)
	=(M+C)(\lambda)+\int_{[\lambda,\lambda+z]}(M+C)'$. Hence, the lemma will follow from showing that $(M+C)(\lambda)$ and subsequently $\mathbbm{1}+[(M+C)(\lambda)]^{-1}\int_{[\lambda,\lambda+z]}(M+C)'$ are invertible. The first will be achieved by checking the definiteness of the sign of $(M+C)(\lambda)$, the second is a consequence of the analytic Fredholm theorem \cite[p. 201]{reedsimon1980}.
	
	By \eqref{eqnantiHerg0} we see that $\Lambda_0(\mu)\leq0$ for any $\mu\leq0$, from which $C(\lambda)\leq0$ follows immediately. Furthermore, the diagonal entries of $M(\lambda)$ are by definition $m_n(\lambda)$; see the discussion preceding \eqref{eqnnthPLP}. Owing to a remark in \cite[p. 4]{marlrozen2009}, for $n=1,2,3,\ldots$, we have $m_n(\lambda)=-i\sqrt{-\lambda}J_n'(i\sqrt{-\lambda})/J_n(i\sqrt{-\lambda})$. We may apply Bessel function properties \cite[Eqs. 10.6.2, 10.19.1]{NIST:DLMF} to show by an algebraic calculation that with fixed $\lambda<0$, as $n\rightarrow\infty$,
	\begin{align}
	m_n(\lambda)=-n\big(1+o(1)\big).
	\end{align}
	We deduce that $(M+C)(\lambda)<0$, and its invertibility follows.
	
	Consider, now, the analytic operator-valued function
	\begin{align}
	\mathscr{A}(z):=[(M+C)(\lambda)]^{-1}\int_{[\lambda,\lambda+z]}(M+C)'\qquad(z\in\mathbb{C}_c).
	\end{align}
	We may apply the reasoning that led to equation \eqref{eqnDtoNNewt} to see that the integral $\int_{[\lambda,\lambda+z]}(M+C)'=(M+C)(\lambda+z)-(M+C)(\lambda)$ is the matrix of a compact operator, for any $z\in\mathbb{C}_c$. Hence, since $[(M+C)(\lambda)]^{-1}$ is bounded, we observe that $\mathscr{A}(z)$ is compact for every $z\in\mathbb{C}_c$. Furthermore, if $z\in\mathbb{C}\setminus\mathbb{R}$ then $\ker\big((M+C)(\lambda+z)\big)=\{0\}$, since if this were not the case we would have a non-trivial function on the interface $\Gi$, meaning (by the Remark on page \pageref{remkereigfn}) there would be an eigenfunction for $T$ with a non-real eigenvalue, which is forbidden by the self-adjointness of $T$. Indeed, this is also contradictory for any $z<0$ since neither $M$ nor $C$ can give rise to eigenvalues less than $\lambda$. Therefore, by the analytic Fredholm theorem \cite[p. 201]{reedsimon1980}, the following two cases are mutually exhaustive:
	
	\noindent (i) $\big(\mathbbm{1}+\mathscr{A}(z)\big)^{-1}$ exists for \emph{no} $z\in\mathbb{C}_{\rm c}$;\quad
	(ii) $\big(\mathbbm{1}+\mathscr{A}(z)\big)^{-1}$ exists for \emph{every} $z\in\mathbb{C}_{\rm c}$.
	
	Clearly for any $z\in\mathbb{C}_c$ we have
	\begin{align}
	(M+C)(\lambda+z)
	&
	=(M+C)(\lambda)\left(\mathbbm{1}+[(M+C)(\lambda)]^{-1}\int_{[\lambda,\lambda+z]}(M+C)'\right)
	\nonumber\\
	&
	=(M+C)(\lambda)\big(\mathbbm{1}+\mathscr{A}(z)\big),\nonumber
	\end{align}
	so if $z<0$ then both sides are invertible. This excludes case (i).
\end{proof}

\section{{Negative eigenvalue asymptotics for operators with singular \textsc{bc}s}}
\label{secfullfrequniq}

	In this section we prove Theorem \ref{thmnegevasympsgen}. Throughout, for convenience, we set $\Lambda=\Lambda_{q,f,\beta}$.
	In the case of symmetric geometry $\Omega=\Omega_1$, 
	owing to the poles of $\Lambda$ being eigenvalues of the operator \eqref{eqnLprimedecomp}, one can sharpen the proof of \cite[Sec. 4]{marlrozen2009} to achieve the following result. A crucial consideration is---as in the proof of Proposition \ref{propLprimespectrum}---that adding $q\neq0$ to the operator leaves the essential spectrum unchanged and perturbs the discrete spectrum by at most $\|q\|_{L^\infty(\Omega_1)}$, \emph{via} \cite[Thms. IV.3.17, V.4.3 \& .10]{kato1995}.
	\begin{lemma}[Marletta--Rozenblum, 2009]
	\label{lemnegevasympssym}
		Take the operator $L'$ defined in \textnormal{\eqref{eqnLprimedecomp}}, and label its eigenvalues by $\lambda^1_n$ so that $\lambda^1_{-1}<0\leq\lambda^1_0$. Then as $n\rightarrow-\infty$ we have
		\begin{equation}
		\label{eqnnegevasymp}
			\lambda^1_n =-{\rm e}^{-2{b}(\thet_0+\tan^{-1}\beta)}{\rm e}^{-2{b} n\pi}\big(1+o(1)\big).
		\end{equation}
		Here $\thet_0=\tan^{-1}(A/B)\in(-\pi/2,\pi/2]$ is known, where
		\begin{equation}
		\label{eqnAB}
			A = \lim\limits_{t\rightarrow+\infty}{\rm e}^{-t}w_s(t), \qquad
			B = \lim\limits_{t\rightarrow+\infty}{\rm e}^{-t}w_c(t),
		\end{equation}
		and the non-trivial functions $w_s$ and $w_c$ satisfy, as $t\rightarrow0$,
		\begin{equation}
		\label{eqnwcws}
			\big[w_s(t),t^{1/2}\sin\big({b}^{-1}\log(t)\big)\big]\rightarrow0,
			\qquad \big[w_c(t),t^{1/2}\cos\big({b}^{-1}\log(t)\big)\big]\rightarrow0
		\end{equation}
		and solve $-w''(t)-(1/4+1/{b}^2)t^{-2}w(t)=-w(t)$ on the half-line $(0,+\infty)$.
	\end{lemma}

	We now prove Theorem \ref{thmnegevasympsgen}, i.e., for the operator $T$ in \eqref{eqnopfulldom} 
	the same asymptotics hold.
	Our approach shows the counting functions of the eigenvalues asymptotically agree. To avoid ambiguity we enumerate the eigenvalues $\lambda_n$ and $\lambda^1_n$ of, respectively, $T$ and $L'$ so that $\lambda_0$ and $\lambda^1_0$ are respectively their smallest non-negative eigenvalues.

	\begin{proof}[Proof of Theorem \ref{thmnegevasympsgen}]
		For discreteness and accumulation points of the spectrum of $T$ we refer to \cite[Sec. 5]{marlrozen2009}. The remainder of the proof is split into two parts, in which we asymptotically bound the counting function for the negative eigenvalues $\lambda_n$ of $T$ from, in turn, above and below by that for the negative eigenvalues $\lambda^1_n$ of $L'$. The lower bound will follow from an asymptotic analysis of
		approximate eigenfunctions of $T$. To find the upper bound we consider the pencil of  \dn operators on $\Gi$ defined in Section \ref{secinterfaceDNops}. As in the proof of Proposition \ref{propLprimespectrum} we apply \cite[Thms. IV.3.17, V.4.3 \& V.4.10]{kato1995} to assume without loss of generality $q=0$.\\
		\\
		\textit{1. Bound from below.} Take any smooth cut-off function $\mu$ on $\Omega$ that is supported and radially symmetric in $\Omega_1$, and takes value $1$ in $\frac{1}{2}\Omega_1$ (in particular, $\pd_\nu\mu\rstr_{\Gamma_1}=0$). Note that the partial derivatives of $\mu$ are supported in the half-annulus $\mathcal A:=\Omega_1\setminus\frac{1}{2}\Omega_1$. Let $n<0$, and choose eigenfunction $\ph_n$ for $L'$ at the eigenvalue $\lambda^1_n$, such that $\|\ph_n\|_{L^2(\Omega_1)}=1$.
		We will show that these $\ph_n$ are \emph{pseudo-modes} for $T$, i.e.,
		\begin{equation}
		\label{eqnpseudomode}
			\frac{\left\|(T-\lambda^1_n)\mu\ph_n\right\|_{L^2(\Omega)}}{\|\mu\ph_n\|_{L^2(\Omega)}}
				=:{\ep}_n\rightarrow0~~~(n\rightarrow-\infty).
		\end{equation}
		By the spectral theorem, denoting the normalised eigenfunctions of $T$ as $\psi_j$ corresponding to $\lambda_j~(j\in\mathbb{Z})$, we may write
		\begin{align}
			\|(T-\lambda^1_n)\mu\ph_n\|^2=\sum_{j\in\mathbb{Z}}(\lambda_j-\lambda^1_n)^2
											|\lan\mu\ph_n,\psi_j\ran_{L^2(\Omega)}|^2.\nonumber
		\end{align}
		Suppose $|\lambda_j-\lambda^1_n|>{\ep}_n$ for every $j\leq-1$. Then ${\ep}_n^2\|\mu\ph_n\|^2_{L^2(\Omega)}>{\ep}_n^2\sum_{j\in\mathbb{Z}}|\lan\mu\ph_n,\psi_j\ran_{L^2(\Omega)}|^2={\ep}_n^2\|\mu\ph_n\|_{L^2(\Omega)}^2$, a contradiction. Thus there is a subsequence $\lambda_{j_n}$ satisfying $|\lambda_{j_n}-\lambda^1_n|\leq{\ep}_n~(n\leq-1)$. Hence if ${\ep}_n\rightarrow0~(n\rightarrow-\infty)$ we will have established the lower bound.
		
		Firstly observe that, by choice of $\mu$ and the fact $\ph_n\in D(L')$, we have
		\begin{align}
			\mu\ph_n-f\pd_\nu(\mu\ph_n)=\mu(\ph_n-f\pd_\nu\ph_n)-f(\pd_\nu\mu)\ph_n=0,\nonumber
		\end{align}
		implying $\mu\ph_n$ is indeed in $D(T)$. Next we calculate that, clearly,
		\begin{align}
			(T-\lambda^1_n)\mu\ph_n=\left\{\begin{array}{ccc}
				0											&\rm{in}&\Omega\setminus\mathcal A, \\
				-(\Delta\mu)\ph_n-2\nabla\mu\cdot\nabla\ph_n&\rm{in}&\mathcal A.
			\end{array}\right.\nonumber
		\end{align}
		If we can show that $\ph_n$ and $\nabla\mu\cdot\nabla\ph_n$ go to $0$ uniformly in $\mathcal A$ then \eqref{eqnpseudomode} will follow.
		
		Set $\kappa_n:=\sqrt{-\lambda^1_n}$, and define the (more conveniently notated) Hankel functions $H^\pm_z=J_z\pm iY_z$, where $J_z$ and $Y_z$ are the Bessel functions of first and second kind. Since $\lambda^1_n<0$ we know that the eigenfunction $\ph_n$ for $L'$ comes from the $L_0'$ operator in the decomposition \eqref{eqnLprimedecomp}. Owing to Proposition \ref{propRn} $\ph_n$ is a constant multiple of
		\begin{align}
		\phi_n(r,\thet)&:={\rm e}^{-\thet/{b}}{\left(H^+_{i/{b}}(ir\kappa_n)H^-_{i/{b}}(i\kappa_n)
							-H^-_{i/{b}}(ir\kappa_n)H^+_{i/{b}}(i\kappa_n)\right)}\nonumber\\
						&=:{\rm e}^{-\thet/{b}}\mathcal H_n(r).\nonumber
		\end{align}

		It follows that, to normalise $\phi_n$ asymptotically, we need to know the leading-order behaviour, as $n\rightarrow-\infty$, of
		\begin{align}
		\label{eqnphinint}
			\frac{1}{{b}\sinh\left(\frac{\pi}{{b}}\right)}\int_{\Omega_1}|\phi_n|^2
					&
					=	\int_0^1|\mathcal H_n(r)|^2r\dee r 
					\nonumber\\
					&
					=\int_0^{\kappa_n}\left|H^+_{i/{b}}(it)H^-_{i/{b}}(i\kappa_n)
												-H^-_{i/{b}}(it)H^+_{i/{b}}(i\kappa_n)\right|^2
													\frac{t\dee t}{\kappa_n^2}.
		\end{align}
		One may easily calculate from Bessel function asymptotics \cite[Eq. 10.17.5--6]{NIST:DLMF} that, for $x>0$,
		\begin{equation}
		\label{eqnhankelasymps}
			H^\pm_{i/{b}}(ix)\sim\sqrt{\frac{2}{\pi}}{\rm e}^{-i(1\pm1)\pi/4\pm\pi/2{b}}{\rm e}^{\mp x}x^{-1/2}\qquad(x\rightarrow+\infty).
		\end{equation}
		Expanding the absolute value in the right-hand side of \eqref{eqnphinint}, we see
		\begin{align}
			\kappa_n^2\frac{\int_0^1|\mathcal H_n(r)|^2r\dee r}{|H^-_{i/{b}}(i\kappa_n)|^2}
				=\int_0^{\kappa_n}|H^+_{i/{b}}(it)|^2&\left\{  1
					+\left|\frac{H^-_{i/{b}}(it)}{H^-_{i/{b}}(i\kappa_n)}\right|^2
					\left|\frac{H^+_{i/{b}}(i\kappa_n)}{H^+_{i/{b}}(it)}\right|^2+\right.
					\nonumber\\
				&\left.
				-2\RE\left(\frac{H^-_{i/{b}}(it)}{H^-_{i/{b}}(i\kappa_n)}
								\frac{H^+_{i/{b}}(i\kappa_n)}{H^+_{i/{b}}(it)}\right)\right\}
									t\dee t.\nonumber
		\end{align}
		By \eqref{eqnhankelasymps}, as $n\rightarrow-\infty$, the term in $\{~\}$ converges pointwise to $1$, and moreover for sufficiently large $n<0$---denoting the greatest such $n$ by $n_0<0$---this term is bounded by $4$. By the latter it follows that for $n\leq n_0$ the integrand of the right-hand side is bounded by $4t|H^+_{i/{b}}(it)|^2$, which, owing to \eqref{eqnhankelasymps}, is certainly integrable over $(0,\infty)$. Hence dominated convergence applies, yielding, as $n\rightarrow-\infty$,
		\begin{align}
			\int_0^1|\mathcal H_n(r)|^2r\dee r
				&\sim\left|H^-_{i/{b}}(i\kappa_n)\right|^2
					\int_0^{\kappa_n}\left|H^+_{i/{b}}(it)\right|^2\frac{t\dee t}{\kappa_n^2} \nonumber\\
				&\sim\left(\int_0^\infty\left|H^+_{i/{b}}(it)\right|^2t\dee t\right)						\kappa_n^{-2}\left|H^-_{i/{b}}(i\kappa_n)\right|^2 \nonumber\\
				&\sim\frac{2\int_0^\infty\left|H^+_{i/{b}}(it)\right|^2t\dee t}{\pi {\rm e}^{\pi/{b}}}\kappa_n^{-3}{\rm e}^{2\kappa_n},\nonumber
		\end{align}
		from which we find
		\begin{align}
			\int_{\Omega_1}|\phi_n|^2
				&\sim\underbrace{\left(\frac{2{b}(1-{\rm e}^{-2\pi/{b}})\int_0^\infty\left|H^+_{i/{b}}(it)\right|^2t\dee t}{\pi}\right)}_{\displaystyle{\hspace{1cm}=:\eta^2}}
					\kappa_n^{-3}{\rm e}^{2\kappa_n} =:c_n^2.\nonumber
		\end{align}
		According to this definition of $c_n$ we have---up to sign---the pointwise asymptotics $\ph_n\sim c_n^{-1}\phi_n$ as $n\rightarrow-\infty$.
		More explicitly, we can calculate from \eqref{eqnhankelasymps} that, pointwise, as $n\rightarrow-\infty$, we have
		\begin{align}
			{\rm e}^{\thet/{b}}\phi_n(r,\thet)&\sim-\frac{2i}{\pi}r^{-1/2}\kappa_n^{-1}{\rm e}^{(1-r)\kappa_n}
			\qquad\Longrightarrow\qquad
			{\rm e}^{\thet/{b}}\ph_n(r,\thet)&\sim-\frac{2i}{\pi\eta}r^{-1/2}\kappa_n^{1/2}{\rm e}^{-r\kappa_n},
			\nonumber
		\end{align}
		and since within $\mathcal A$ we have $1/2<r<1$ it is clear that $\ph_n\rstr_{\mathcal A}\rightarrow0$ uniformly as $n\rightarrow-\infty$.
		
		Now we examine
		\begin{align}
			\nabla\mu(r,\thet)\cdot\nabla\ph_n(r,\thet)
				&=|\nabla\mu(r,\thet)||\nabla\ph_n(r,\thet)|
					\cos\big(\arg\nabla\mu(r,\thet)-\arg\nabla\ph_n(r,\thet)\big) \nonumber\\
				&=|\pd_r\mu(r,\thet)\pd_r\ph_n(r,\thet)|
					\cos\left(\frac{\pd_\thet\ph_n(r,\thet)}{r}\right).\nonumber
		\end{align}
		It is clear from our prior calculations that $\pd_r\ph_n\sim c_n^{-1}\pd_r\phi_n$. Now recall \cite[Eq. 10.6.2]{NIST:DLMF} that $\frac{\dee}{\dee z}H^\pm_\zeta(z)=\frac{\zeta}{z}H^\pm_\zeta(z)-H^\pm_{\zeta+1}(z)$,
		from which we derive
		\begin{align}
		\label{eqnrderivphin}
			 {\rm e}^{\thet/{b}}\pd_r\phi_n(r,\thet)=\hspace{3pt}
				&\frac{i}{r{b}}\left(H^+_{i/{b}}(ir\kappa_n)H^-_{i/{b}}(i\kappa_n)
					-H^-_{i/{b}}(ir\kappa_n)H^+_{i/{b}}(i\kappa_n)\right)
					\nonumber\\
				&-i\kappa_n\left(H^+_{1+i/{b}}(ir\kappa_n)H^-_{i/{b}}(i\kappa_n)
					-H^-_{1+i/{b}}(ir\kappa_n)H^+_{i/{b}}(i\kappa_n)\right).
		\end{align}
		We calculate \cite[Eq. 10.17.5-6]{NIST:DLMF} that $H^\pm_{1+i/{b}}(ix)\sim\sqrt{\frac{2}{\pi}}{\rm e}^{-i(1\pm3)\pi/4\pm\pi/2{b}}{\rm e}^{\mp x}x^{-1/2}$, from which we easily show that the first and second terms in the right-hand side of \eqref{eqnrderivphin} are asymptotically equivalent respectively to
		\begin{equation}
			\frac{2}{r\pi{b}}{\rm e}^{(1-r)\kappa_n}r^{-1/2}\kappa_n^{-1} \qquad{\rm and }\qquad
				-\frac{2}{\pi}{\rm e}^{(1-r)\kappa_n}r^{-1/2}.
		\end{equation}
		It follows, after substituting these into \eqref{eqnrderivphin} then dividing by $c_n$, that
		\begin{equation}
			{\rm e}^{\thet/{b}}\pd_r\ph_n(r,\thet)\sim-\frac{2}{\pi\eta}r^{-1/2}\kappa_n^{3/2}{\rm e}^{-r\kappa_n}\qquad(n\rightarrow-\infty),
		\end{equation}
		and therefore $\nabla\mu\cdot\nabla\ph_n$ must go to $0$ uniformly in $\mathcal A$. The lower bound on the difference between the counting functions for the negative eigenvalues of $T$ and $L'$ follows.\\
		\\
		\textit{2. Bound from above.} Here we will use the pencil of \dn operators on $\Gi$ from the sub-domains either side of the interface, and analyse non-triviality of its kernel, which occurs precisely when $T$ has an eigenvalue.

		The normalisation $h_0=1$ (from the discussion following the Remark on p. \pageref{remkereigfn}) ensures $\lambda\ll0$ is an eigenvalue if and only if $\exists\vh=(h_1,h_2,\ldots)^T$ with
		\begin{align}
		\label{eqn1h}
			\left(\begin{array}{c|c}
				m_0(\lambda)+a(\lambda)	&	\vb(\lambda)^T \\
				\hline
				\vb(\lambda)		&	M(\lambda)+C(\lambda)
			\end{array}\right)
			\left(\begin{array}{c}
			1 \\ \vh
			\end{array}\right)
			=0.
		\end{align}
		After expanding the product we find that this can only happen if $\vb(\lambda)=-\big(M(\lambda)+C(\lambda)\big)\vh$. Owing to Lemma \ref{lemcorolanalyticFredholm} we know $M(\lambda)+C(\lambda)$ is invertible for $\lambda\ll0$. Define $\vh(\lambda)=-\big(M(\lambda)+C(\lambda)\big)^{-1}\vb(\lambda)$ and set $h(\lambda)={1\choose{\vh(\lambda)}}$. Substituting $h$ into \eqref{eqn1h} and then taking a quadratic form we observe that $\lambda\ll0$ is an eigenvalue for $T$ if the following expression vanishes:
		\begin{align}
		\mathcal{E}(\lambda)
			&:=\left\lan\big(\Lambda_1(\lambda)+\Lambda_0(\lambda)\big)h(\lambda),
				h(\lambda)\right\ran_{L^2(\Gi)}
				\nonumber\\
			&=m_0(\lambda)+a(\lambda)-\vb(\lambda)^T\big(M(\lambda)+C(\lambda)\big)^{-1}\vb(\lambda).
			\nonumber
		\end{align}
		
		Thanks to Lemma \ref{lemHerg} both $\Lambda_j$ are Herglotz in quadratic form, and analytic as operator-valued functions; in particular any sub-block is analytic. Furthermore $\Lambda_0$ arises from a problem on a bounded domain with regular \textsc{bc}s. Therefore we see that $a(\lambda)$, $\vb(\lambda)$ and $C(\lambda)$ lack poles when $\lambda$ is sufficiently negative or non-real. For the same $\lambda$ the coefficient $a(\lambda)$ is never $0$, $\vb(\lambda)$ is not identically the zero vector---though it could have some zero entries---and $C(\lambda)$ is never null. Moreover the Herglotz property of $\Lambda_1+\Lambda_0$ ensures that $\IM\big(\mathcal E(\lambda)\big)\IM(\lambda)\geq0$.
		The invertibility of $M(\lambda)+C(\lambda)$ for non-real $\lambda$ follows from Lemma \ref{lemcorolanalyticFredholm}, and is enough to ensure analyticity of $\mathcal{E}$ away from $\mathbb{R}$, and establish that $\mathcal E$ is, like $\Lambda_1+\Lambda_0$, Herglotz. Now, for $\lambda\ll0$ $M(\lambda)$ has fixed sign, so we see that the poles of $\mathcal E$ and $m_0$ are identical; since poles and zeros of Herglotz functions interlace we see that between any two sufficiently negative poles of $m_0$ there is a zero of $\mathcal E$ and hence at most one eigenvalue of $T$. The upper bound
		 follows.
	\end{proof}

\section{Discussion}
\label{secdiscussion}

	A strong motivation for this investigation lies in the question: \emph{What sort of \bc can we expect on the inaccessible portion $\Gc$ of the boundary $\pd\Omega$?} Considering the inverse conductivity problem for \eqref{eqncond}, a Neumann condition on $\Gc$ corresponds to a ``perfect insulator'' (no electric potential flux across the boundary), whilst a Dirichlet condition corresponds to a ``perfect conductor''. Clearly, at least in problems paralleling most physical scenarios, \emph{one cannot expect a pure Dirichlet or Neumann condition}.
	
	The only (linear first-order) alternative is a Robin condition with an unknown Dirichlet-to-Neumann ratio $f$. In this situation \cite{berrydennis2008,marlrozen2009} tell us we need to be wary of points at which $f$ vanishes locally linearly. But as just established, we can exploit this type of singular \bc to recover $f$ in a neighbourhood of its zero, and subsequently our self-adjointness \bc and the Schr\"odinger potential.
	
	Our methods do not generalise to further singularities, since it would no longer be clear from which of the $b$-parameters a given negative eigenvalue had arisen.

	Nevertheless there are various possible further routes.
	\begin{itemize}
		\item \emph{Generalise the admissible class of $f$.} In the above, $f$ must be linear in a neighbourhood of $0$. Given the results suggested in \cite{berrydennis2008,marlrozen2009,berry2009}, it should be possible to prove that a general $f$ with finitely (possibly countably) many simple zeros yields, from \eqref{eqnspectralSchro}, an operator with both symmetric adjoint and self-adjoint restriction, sufficient to generalise Inverse Problem \ref{invqfbeta}.
		\item \emph{Change the shape of $~\supp(f)=\Gamma_1$ (equivalently $\Omega_1$).} Expressing a general $\Gamma_1\in C^1$ in the form of a perturbation of a straight line is a possible route towards this.
		\item \emph{Reduce the number of boundary data used.} Inverse Problem \ref{invqfbeta} requests recovery of the triple $(q,f,\beta)$, which has 3 variables. The Schwartz kernel of the \dn map at a fixed $\lambda\in\mathbb{R}$ has 2 variables. Thus it should suffice for full uniqueness to know the \dn map at just two points in $\mathbb{R}$---indeed, it should be over-determined.
	\end{itemize}

\section*{Acknowledgements}

I am grateful to the \textsc{epsrc} and Cardiff University for respectively funding and hosting my P\textsc{h}.D. studies during which I proved the above results, as well as my supervisors Professors Marco Marletta and Malcolm Brown for their indispensable guidance.

\bibliographystyle{acm}
\bibliography{Freddy_references}%

\end{document}